\documentclass[12pt]{elsarticle}

\usepackage{lineno,hyperref,afterpage}
\usepackage{amsmath,amssymb}
\usepackage{xcolor}
\usepackage{amsthm} 
\usepackage[shortlabels]{enumitem}
\usepackage[scale=0.7]{geometry} 
\usepackage{verbatim}
\usepackage[nottoc]{tocbibind}

\usepackage{mathtools}
\usepackage{titlesec}
\titleformat{\section}[block]{\bfseries\filcenter}{}{1em}{}

\DeclarePairedDelimiter\floor{\lfloor}{\rfloor}

\DeclareRobustCommand{\SkipTocEntry}[5]{}

\allowdisplaybreaks
\modulolinenumbers[5]

\newtheorem{The}{Theorem}[section]
\newtheorem{Lem}[The]{Lemma}

\newtheorem{Rem}[The]{Remark} 
\newtheorem{Def}[The]{Definition}
\newtheorem{Pro}[The]{Proposition}

\newcommand{\R}{{\mathbb R}}
\newcommand{\Z}{{\mathbb Z}}
\newcommand{\N}{{\mathbb N}}

\newcommand{\eps}{\varepsilon}

\newcommand{\dimens}{n}
\newcommand{\norm}[1]{\left\lVert #1 \right\rVert}
\newcommand{\abs}[1]{\left\lvert #1 \right\rvert}
\newcommand{\aabs}[1]{\left\lVert #1 \right\rVert}
\newcommand{\ip}[2]{\left\langle #1,#2 \right\rangle}

\usepackage{natbib}
\begin{document}

\begin{center}
    \huge{Uniqueness in an inverse problem of fractional elasticity}
    \vspace{2mm}
    
    \large{Giovanni Covi, Maarten de Hoop, Mikko Salo}
\end{center}

\section*{Abstract}
    \noindent We study an inverse problem for fractional elasticity. In analogy to the classical problem of linear elasticity, we consider the unique recovery of the Lam\'e parameters associated to a linear, isotropic fractional elasticity operator from fractional Dirichlet-to-Neumann data. In our analysis we make use of a fractional matrix Schr\"odinger equation via a generalization of the so-called Liouville reduction, a technique classically used in the study of the scalar conductivity equation. We conclude that unique recovery is possible if the Lam\'e parameters agree and are constant in the exterior, and their Poisson ratios agree everywhere. Our study is motivated by the significant recent activity in the field of nonlocal elasticity.

\tableofcontents

\section{Introduction}

We consider an inverse problem for a fractional elasticity operator of the form $$\mathbf E^s u := (\nabla\cdot)^s(\mathcal C(x,y)\nabla^s u)$$ for a fixed $s\in (0,1)$. The fractional gradient $\nabla^s$ and the fractional divergence $(\nabla\cdot)^s$ appearing in the definition of $\mathbf E^s$ 
go back to \cite{DGLZ12} and they will be defined in detail in Section \ref{sec-prel}. They can be thought of as nonlocal counterparts of the classical gradient and divergence operators, with the expected properties $(\nabla^s)^* = (\nabla\cdot)^s$ and $(\nabla\cdot)^s\nabla^s = (-\Delta)^s$. The tensor $\mathcal C(x,y):= C^{1/2}(x)C^{1/2}(y)$ is itself a nonlocal counterpart of the classical stiffness tensor $C$. Section \ref{sec-operator} is dedicated to the detailed definition of $\mathbf E^s$. 

\begin{center}
\includegraphics[height=6cm]{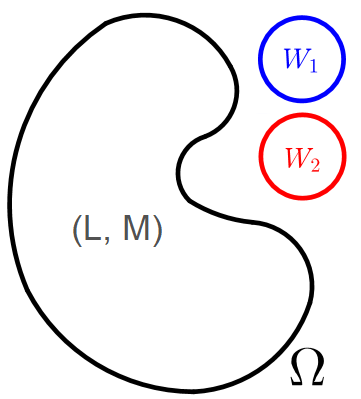}
\end{center}

We make the assumption that $C$ is isotropic, that is, its properties are completely described by two Lam\'e parameters $L,M$ (see Section \ref{subsec-elast}). The scalar functions $L,M$ are assumed to be constant outside a domain of interest $\Omega$, and unknown within it. The goal of the inverse problem we wish to study is to recover $L,M$ (and thus $C$) within $\Omega$ from measurements performed on its exterior $\Omega_e := \mathbb R^n\setminus\overline\Omega$. To this end, we start by considering the direct problem
\begin{equation*}
    \begin{array}{rll}
       \mathbf E^s u & =0& \quad \mbox{ in } \Omega \\
       u & = f & \quad \mbox{ in } \Omega_e
    \end{array},
\end{equation*}
for which we prove well-posedness in a weak sense for any sufficiently regular exterior datum $f$. This means that given any $f$ defined on $\Omega_e$ (e.g. we might have $f\in C^\infty_c(\Omega_e)$) there is one and only one solution $u_f\in H^s(\mathbb R^n)$ to the direct problem. Building on this, we define the exterior measurements as a nonlocal Dirichlet-to-Neumann (DN) map $\Lambda_{L,M}$, associating each exterior datum $f$ to the corresponding nonlocal Neumann data $\mathbf E^s u_f|_{\Omega_e}$. It is clear that $\Lambda_{L,M}$ carries information about the Lam\'e parameters $L,M$, as these are involved in the definition of the unique solution $u_f$. More specifically, we would like to recover $L,M$ within $\Omega$ from partial data, that is, from the knowledge of $\Lambda_{L,M} f|_{W_2}$ for all $f\in C^\infty_c(W_1)$, where $W_1, W_2$ are non-empty, open and disjoint subsets of $\Omega_e$. This means that the exterior data $f$ will be supported in $W_1$, while the measurements will be performed on $W_2$ only. These restrictions represent the physical situation in which not all of the exterior is accessible for measurement. We thus ask the following inverse problem:
\vspace{4mm}

\textbf{Q:} \emph{Does $\Lambda_{L_1,M_1}f|_{W_2}=\Lambda_{L_2,M_2}f|_{W_2}$ for all $f\in C^\infty_c(W_1)$ imply that $L_1=L_2$ and $M_1=M_2$ within $\Omega$?}
\vspace{4mm}

We build on the work of \cite{GSU20}. We will first use the assumption on the DN maps in order to obtain an integral identity, the so called Alessandrini identity, relating the difference of the DN maps to the differences of the Lam\'e parameters via some special solutions to the direct problem. Then, we shall test the Alessandrini identity with aptly chosen solutions in order to deduce the desired result. Such solutions will be produced by means of a Runge approximation property, which we will prove for our equation. The proof of the Runge approximation property will itself rely on a unique continuation property (UCP), which is the key point of the technique.  

However, as it will be clarified in Remark \ref{why-though}, the approach described above does not directly lead to the desired result. It is instead necessary to first reduce the given problem to a more manageable one of Schr\"odinger type. We do so by means of the so called \emph{fractional Liouville reduction}, which is reminiscent of the techniques used for both the classical and fractional conductivity equations (see \cite{C20} for the latter). After the reduction, the direct problem reads
\begin{equation*}
    \begin{array}{rll}
       (-\Delta)^{s-1}Dw -w\cdot Q & = 0 & \quad \mbox{ in } \Omega \\
       w & = g & \quad \mbox{ in } \Omega_e
    \end{array},
\end{equation*}
where $D$ is a fixed differential operator of second order, $Q$ is a new potential containing information relative to the Lam\'e parameters $L,M$, the new exterior datum $g$ is computed from $f$, and the new solution $w$ is computed from $u$ (see Section \ref{sec-liou} for the details).

The reduction from the fractional elasticity equation to matrix Schr\"odinger equation is possible for arbitrary Lam\'e coefficients. However, in order to make a reduction on the level of exterior measurements and to derive a suitable integral identity, we need to assume that the Poisson ratio is a fixed (unknown) function:

\begin{Def}
If $(L_1,M_1), (L_2,M_2)$ are two couples of Lam\'e parameters, we write $(L_1,M_1)\sim (L_2,M_2)$ if and only if they coincide on $\Omega_e$ and $\nu_1=\nu_2$ on $\R^n$, where the Poisson ratio $\nu$ relative to a couple of Lam\'e parameters $(L,M)$ is $\nu:= \frac{L}{(n-1)L+2M}$ (see Section \ref{subsec-elast} for the precise definition).
\end{Def}

As it turns out, if the Poisson ratio is a fixed function, then the technique described above will work for the transformed problem and will allow us to determine the matrix potential $Q$ from the nonlocal DN map. We can then recover the Lam\'e parameters $(L,M)$ from $Q$ and prove the following uniqueness result.

\begin{The}\label{main-theorem}
Let $\Omega, W_1, W_2\subset\mathbb R^n$ be bounded open sets such that $W_1, W_2\subset\Omega_e$, and assume $s\in(0,1)$. Let $(L_1, M_1)$ and $(L_2, M_2)$ be two couples of Lam\'e parameters satisfying assumptions (A1)-(A3). If $(L_1,M_1)\sim (L_2,M_2)$ and 
$$ \Lambda_{L_1,M_1}f|_{W_2} = \Lambda_{L_2,M_2}f|_{W_2} \qquad \mbox{ for all } f\in C^\infty_c(W_1),$$
then $(L_1,M_1)=(L_2,M_2)$. 
\end{The}

Here assumptions (A1)-(A3) state that $L,M$ are $(2s+\eps)$-H\"older regular, constant outside of $\Omega$ and enjoy a certain positivity condition (see Section \ref{sec-operator}). 

\addtocontents{toc}{\SkipTocEntry}
\subsection{Motivation and connection to the literature}

 We consider nonlocal elasticity, that is, space-nonlocality, when
nonlocal stress is defined as the Riesz fractional integral of the
strain field in space. There exist several complex phenomena,
occurring possibly in damage zones \cite{BenZion}, that cannot be
addressed by classical local continua. For a recent overview, we refer
to \cite{Failla-Zingales-2020}. While the notion of nonlocal
elasticity dates back to the work of Mindlin \cite{Mindlin-1964,
  Mindlin-1965} (who considered a variational formulation resulting in
stable and well-posed solutions of boundary value problems, enabling
the removal of singularities from dislocations and cracks), we focus
on a fractional generalization of the classical Eringen integral model
\cite{Eringen-1972, Er02} of nonlocality. That is, we
consider nonlocal elasticity based on the so-called fractional (or
generalized) linear ``gradient''-elasticity model
\cite{CarpinteriCornettiSapora-2011, TarasovAifantis-2018}
\begin{equation*} \label{eq:gradela}
   \boldsymbol{\sigma} = C : [\boldsymbol{\epsilon}
               + a_{s} (- \Delta)^{s}
               \boldsymbol{\epsilon}] ,\quad
   a_{s} = \ell_{s}^2 ,
\end{equation*}
where $(- \Delta)^{s}$ is a fractional Laplacian in the
Riesz form, $s \in \mathbb R^+\setminus \mathbb Z$, $\ell_{s}$
is a material parameter signifying an internal length scale, $\boldsymbol{\sigma}$ is the stress tensor, and $\boldsymbol{\epsilon}$ is the strain tensor. Assuming that the stiffness tensor $C$ is isotropic and has constant coefficients results in the following fractional elasticity operator $ \mathbf T^s$ from \cite{TarasovAifantis-2018}:
$$ \mathbf T^s u := \nabla\cdot (C (1+a_s(-\Delta)^s)\nabla u) = \nabla\cdot (C \nabla u) + a_s\nabla\cdot (C (-\Delta)^s\nabla u).$$
The operator $\mathbf E^{s}$ which we study in the present paper generalizes $\mathbf T^{s}$ to the case of variable coefficients (see section \ref{operators} for the precise definitions of the fractional operators used here). In order to see this, compute $$ \mathbf E^{s+1} u := (\nabla\cdot)^{s+1} (C^{1/2}(x):C^{1/2}(y)\nabla^{s+1} u) = \nabla\cdot\{ (\nabla\cdot)^s C^{1/2}(x):C^{1/2}(y) \nabla^s\nabla u \},$$
where the relations $\nabla^{s+1}=\nabla^s\nabla$ and the corresponding one for the fractional divergence are definitions, in accordance to \cite{CMR20}. If the coefficients are assumed to be constant, a straightforward computation shows that $$\mathbf E^{s+1} u = \nabla\cdot( (\nabla\cdot)^s C \nabla^s\nabla u )= \nabla\cdot( C(\nabla\cdot)^s \nabla^s\nabla u ) = \nabla\cdot( C(-\Delta)^s \nabla u ),$$
and thus
$$ \mathbf T^s u = a_s \mathbf E^{s+1}u + \nabla\cdot (C\nabla u) .$$

In an alternative introduction of a model for fractional elasticity directly through a kernel,
one develops a fractional Taylor series using the Caputo fractional
derivative (involving a left-sided Riemann-Liouville fractional
integral) of its Fourier transform \cite{OdibatShawagfeh-2007,
  KilbasSrivastavaTrujillo-2006}. The physical basis of such an
introduction is an assumption pertaining to fractional spatial
dispersion for a nonlocal elastic continuum.

Zorica and Oparnica \cite{ZoricaOparnica2020} presented
time-fractional wave equations that model hereditary viscoelastic
behavior and space-fractional wave equations associated with certain
nonlocal elasticity models. For a number of fractional wave
equations, the authors provided mathematical evidence of energy
dissipation and conservation.

An inverse problem with nonlocal elasticity was considered by Askes
and Aifantis \cite{AskesAifantis}. For an important overview of
general fractional derivative equations, containing fractional time
and space derivatives, and inverse problems examining ill-posedness in
space dimension $1$, we refer to Jin and Rundell
\cite{JinRundell2015}.
\vspace{3mm}

The theory of linear elasticity has given rise to classical inverse problems which are related to our question, and can rather be considered the main inspiration of our study. We briefly present the main concepts of the classical theory of linear elasticity in Section \ref{subsec-elast} (see also \cite{landau-elasticity}).  The operator of classical elasticity is
$$\mathbf E u := \nabla\cdot( C\nabla u),$$
where $C$ is the stiffness tensor. Since $\mathbf E$ is a local operator, the measurements involved in the inverse problems are performed on the boundary $\partial\Omega$ of the domain $\Omega$ rather than in its exterior $\Omega_e$. 
Most of the results concern the isotropic case where $C$ depends on two scalar Lam\'e parameters $L$ and $M$. The main uniqueness results for this inverse problem in three dimensions are in \cite{NU1994, NU2003} and \cite{ER2002, ER2004}. They state that if $M$ is sufficiently close to a constant, or if either $L$ or $M$ is a fixed (unknown) function, then the Lam\'e parameters are uniquely determined by the boundary measurements. In \cite{ER2004} one also finds the classical analogue of Theorem \ref{main-theorem} stating that the Lam\'e parameters are determined if the Poisson ratio is a fixed function. The two-dimensional case is considered in \cite{IY2015}.  Many other different aspects of this inverse problem have been considered, such as inclusion detection \cite{Ela5}, identification of the elastic moduli in beams, plates and other geometric configurations \cite{Ela12,Ela48,ER2002,Ela70}, the linearized problem \cite{Ela81,Ela82,Ela83}, uniqueness \cite{Ela11,NU1994,NU2003}, and identification of residual stresses \cite{Ela9,Ela67,Ela106,Ela127}, among others. We refer to the survey \cite{survey} for many more results.
\vspace{3mm}

The inverse problem we study is related to the fractional Calder\'on problem, which was introduced in the seminal paper~\cite{GSU20} as a nonlocal counterpart to the classical Calder\'on problem arising in electric impedance tomography. Uniqueness was achieved in the case of bounded potentials and fractional exponent $s\in(0,1)$~\cite{GSU20} also with a single measurement \cite{GRSU20}, and later extended to rough potentials~\cite{RS17} and all positive fractional exponents $s\in\mathbb R^+\setminus\mathbb Z$~\cite{CMR20}. Perturbed versions of the same problem were studied e.g. in~\cite{CLR18} for $s\in (1/2, 1)$ and first order perturbations and \cite{CMRU20} for $s\in\mathbb R^+\setminus\mathbb Z$ and general high order local perturbations. Nonlocal perturbations were studied in some specific cases in \cite{BGU21, C21, LO21}. Uniqueness has also been studied in numerous other settings, including the fractional magnetic Schr\"odinger equation~\cite{C20a, CMR20, Li20a, Li20b, Li21} and the fractional heat equation~\cite{LLR19,RS17a}. A very recent fractional elasticity equation with constant principal coefficients \cite{LL21} and the fractional conductivity equation~\cite{C20} bear a strong connection with the present study. Moreover, the fractional Schr\"odinger equation was studied in the semilinear setting~\cite{LL19, Li20a, Li20b, Li21}. We refer to the surveys~\cite{S17,R18} for more information about the fractional Calder\'on problem.

\addtocontents{toc}{\SkipTocEntry}
\subsection{Organization of the rest of the article}
The remaining part of the paper is organized as follows. Section \ref{sec-prel} contains preliminaries from functional analysis, classical elasticity theory and nonlocal vector calculus, as well as the definitions of the notations used in the article. Section \ref{sec-operator} defines and describes the main object of study of the paper, the fractional elasticity operator $\mathbf E^s$. The Dirichlet problem related to such operator is studied in Section \ref{sec-wellposedness}, and in Section \ref{sec-liou} it is shown to be equivalent to a Dirichlet problem for the fractional Schr\"odinger equation. An integral identity relating coefficients and measured data, the so called Alessandrini identity, is shown to hold in Section \ref{sec-alex}. In Section \ref{sec-proof} we prove the Runge approximation property and eventually the main Theorem \ref{main-theorem}. 

\subsection*{Acknowledgments}

Giovanni Covi was supported by an Alexander-von-Humboldt postdoctoral fellowship. Maarten de Hoop was supported by the Simons Foundation under the MATH + X program, the National Science Foundation under grant DMS-2108175, and the corporate members of the Geo-Mathematical Imaging Group at Rice University. Mikko Salo was partly supported by the Academy of Finland (Centre of Excellence in Inverse Modelling and Imaging, grant 284715) and by the European Research Council under Horizon 2020 (ERC CoG 770924).

\section{Preliminaries}\label{sec-prel}
In this section we recall the definitions of relevant function spaces, define some concepts from mathematical physics related to the problem of elasticity, and establish some useful notations.

\addtocontents{toc}{\SkipTocEntry}
\subsection{Tensor products and contractions}
We will make wide use of the concepts of tensor product and contraction, mainly with respect to vectors. This will let us write our equations in a more understandable way. We let $\mathbb N$ be the set of \emph{strictly} positive integers.

\begin{Def}[Tensor product]\label{def-tensor}
Let $m,n \in \mathbb N$, and assume $a, b$ are multi-indices belonging to $\mathbb N^m, \mathbb N^n$ respectively. Consider two tensors $A_{\alpha_{1}, ... , \alpha_{m}}$ and $B_{\beta_{1}, ... , \beta_{n}}$, with $\alpha_i \in \{1, ..., a_i\}$ for $i=1, ..., m$ and $\beta_j \in \{1, ..., b_j\}$ for $j=1, ..., n$ respectively. The \emph{tensor product} $A\otimes B$ is the new tensor of elements
$$(A\otimes B)_{\alpha_1, ..., \alpha_m, \beta_1, ..., \beta_n} = A_{\alpha_{1}, ... , \alpha_{m}}B_{\beta_{1}, ... , \beta_{n}}.$$
In particular, if $m=n=1$ (and thus $A\in \R^a, B\in \R^b$ are vectors) the tensor product $A\otimes B$ is just the $a\times b$ matrix of elements $$(A\otimes B)_{\alpha,\beta} = A_\alpha B_\beta, \quad \mbox{for } \alpha = 1,..., a \mbox{ and } \beta = 1, ..., b.$$
\end{Def}

\begin{Def}[Tensor contraction of order $k$]
Let $m,n,a,b,A,B$ be as in Definition \ref{def-tensor}. Assume that $k \in \mathbb N$ is such that $k\leq \min \{m,n\}$, with $a_{m+\ell-k} = b_{\ell}$ for all $\ell \in \{1, ..., k\}$. We define the \emph{$k$-th contraction} of tensors $A$ and $B$ as the new tensor given by
$$ (A\cdot_k B)_{\alpha_1, ..., \alpha_{m-k}, \beta_{k+1}, ..., \beta_n} = A_{\alpha_1, ..., \alpha_{m-k}, \gamma_1, ..., \gamma_k} B_{\gamma_1, ..., \gamma_k, \beta_{k+1}, ..., \beta_n} ,$$
where we assume the Einstein summation convention on repeated indices. The ranges of the $\alpha,\beta$ indices appearing in the above formula are the same as in the definitions of $A,B$, while the $\gamma_\ell$ index has range in $\{ 1, ..., b_\ell \}$ for $\ell= 1, ..., k$. 
\end{Def}

 In the present work, we apply the above definition only with $k=1$ and $k=2$, so we use the standard symbols $\cdot$ and $:$ in place of the more general $\cdot_1$ and $\cdot_2$. It is useful to observe that when $A,B$ are vectors or matrices the above contraction operator $\cdot$ coincides with the usual scalar product and matrix multiplication.
 \vspace{3mm}
 
 The definitions are immediately extended to functions. If $A,B$ are tensor-valued functions defined on some set $\Omega$, then we let \begin{align*}
     A\otimes B : x &\mapsto A(x)\otimes B(x), \quad \mbox{for all } x\in\Omega, \\ A\cdot_k B : x &\mapsto A(x)\cdot_k B(x) , \quad \mbox{for all } x\in\Omega.
 \end{align*}   Accordingly, if $V,W$ are sets of tensor-valued functions defined on some set $\Omega$, then we let $V\otimes W$ be the new set of functions on $\Omega$ given by
$$V\otimes W := \{ v\otimes w , v\in V, w\in W \}.$$

\begin{Rem}
This should not be confused with the familiar Cartesian product $V\times W$. For example, if $a,b\in\N$ and $V,W$ are sets of vector-valued functions $v:\Omega\mapsto \R^a$ and $w:\Omega\mapsto\R^b$ respectively, then the elements of $V\times W$ map $\Omega$ to $\R^{a+b}$, while the ones of $V\otimes W$ map $\Omega$ to $\R^{a\times b}$. In the present work we shall often have $a=2$ and $V=\{v\}$ for some fixed function $v: \Omega \rightarrow \mathbb R^2$. 
\end{Rem}

The following Lemma collects some elementary properties of tensor products and contractions, which can be easily proved using the index notation:

\begin{Lem}\label{product-properties}
Let $m,n,p \in \N$, and assume that $a,b,c$ are multi-indices belonging to $\N^m, \N^n, \N^p$ respectively. Consider three tensors \begin{align*}
    A_{\alpha_1, ..., \alpha_m},& \quad \mbox{ with } \alpha_i\in \{1, ..., a_i\} \mbox{ for } i=1,...,m, \\ 
    B_{\beta_1, ..., \beta_n},& \quad \mbox{ with } \beta_j\in \{1, ..., b_j\} \mbox{ for } j=1,...,n, \\
    C_{\gamma_1, ..., \gamma_p},& \quad \mbox{ with } \gamma_k\in \{1, ..., c_k\} \mbox{ for } k=1,...,p. 
\end{align*} 
The following equalities hold whenever the tensor contractions are well-defined:
\begin{enumerate}[(i)]
    \item $(A\otimes B)\cdot_n C = (A\otimes C)\cdot_n B \quad \mbox{ for } n=p,$
    \item $A\cdot_k (B\otimes C) = (A\cdot_k B)\otimes C \quad \mbox{ for all } k\in\N,$
    \item $(A\cdot_k B)\cdot_q C = A\cdot_k (B\cdot_q C) \quad \mbox{ for all }  k,q\in\N \mbox{ with } n\geq k+q, $
    \item $A\cdot_{n+p}(B\otimes C) = (A\cdot_p C)\cdot_n B$.
\end{enumerate} 
If $A,B$ are tensor-valued functions of $x\in\Omega$, then the equalities
\begin{enumerate}[(i)]\setcounter{enumi}{4}
    \item $\nabla(A\otimes B) = \nabla A \otimes B + A\otimes \nabla B\quad \mbox{ for } m=a=1$,
    \item $\nabla(A\cdot_m B) = \nabla A\cdot_m B + \nabla B\cdot_m A\quad \mbox{ for }  m=n $,
    \item $\nabla\cdot (A\otimes B) = A\otimes\nabla\cdot B + B^T\cdot \nabla A \quad \mbox{ for }  m=a=1,\; n=2 $. 
\end{enumerate} 
also hold whenever the tensor contractions are well-defined. \hfill$\square$ 
\end{Lem}

\begin{Rem}
For all vectors $v$ we use the convention $(\nabla v)_{ij}:= \partial_iv_j$.
\end{Rem}

\addtocontents{toc}{\SkipTocEntry}
\subsection{Fractional Sobolev spaces}
Let $r\in\mathbb R$, and assume that $\Omega, F\subset\mathbb R^n$ respectively are an open and a closed set. We indicate by $H^{r}(\mathbb R^n)$ the usual $L^2$-based Bessel potential spaces. These are endowed with the norm
$$\|u\|_{H^{r}} := \|(1+|\xi|^2)^{r/2}\hat u(\xi)\|_{L^2(\R^n)}.$$

\noindent Here by $\hat u(\xi) = \mathcal{F}u(\xi) := \int_{\mathbb R^n} e^{-ix\cdot\xi}u(x)\,dx$ we indicate the Fourier transform. Using the same notation as in \cite{McLean} and \cite{GSU20}, we also define the following fractional Sobolev spaces:
\begin{align*}
    H^{r}(\Omega) & := \{ u|_\Omega , u\in H^{r}(\mathbb R^n) \} ,\\
    \widetilde H^{r}(\Omega) & := \,\mbox{closure of } C^\infty_c(\Omega) \mbox{ in } H^{r}(\mathbb R^n), \\
    H^{r}_F(\mathbb R^n) & := \{u\in H^{r}(\mathbb R^n) : \mbox{ supp}(u)\subset F\}. 
\end{align*}
\noindent In particular, $H^{r}(\Omega)$ is endowed with the quotient norm $\|u\|_{H^{r}(\Omega)}:= \inf\{ \|w\|_{H^{r}} : w\in H^{r}(\mathbb R^n), w|_{\Omega} = u \}$. One sees that the following inclusions and identities
$$ \widetilde H^r(\Omega) \subseteq H^r(\Omega)\,,\quad \widetilde H^r(\Omega) \subseteq H^r_{\overline\Omega}\,,\quad H^r_F \subseteq H^r\,,$$
$$ (\widetilde H^r(\Omega))^* = H^{-r}(\Omega)\,,\quad (H^r(\Omega))^* = \widetilde H^{-r}(\Omega) $$
hold. If in addition $\Omega$ is known to be Lipschitz, then $\widetilde H^r(\Omega) = H^r_{\overline\Omega}$ for all $r\in \mathbb R$, as shown in \cite{CMRU20, McLean}.

We shall also use products of Sobolev spaces. If $r\in \mathbb R$, the space $H^r(\mathbb R^n)\times H^r(\mathbb R^n)$ consists of all the vectors $u=(u_1,u_2)$ with $u_1,u_2\in H^r(\mathbb R^n)$, endowed with the norm
$$\|u\|^2_{H^r(\mathbb R^n)\times H^r(\mathbb R^n)} := \|u_1\|^2_{H^r(\mathbb R^n)} + \|u_2\|^2_{H^r(\mathbb R^n)}.$$
\\

\addtocontents{toc}{\SkipTocEntry}
\subsection{The fractional Laplacian and other nonlocal operators}\label{operators}
Let $s\in\mathbb R^+\setminus\mathbb Z$ and $u\in\mathcal S$, the set of Schwartz functions. The fractional Laplacian of $u$ can be defined as
$$(-\Delta)^s u := \mathcal{F}^{-1}(|\xi|^{2s}\hat u(\xi)),$$
which makes $(-\Delta)^s$ a continuous map from $\mathcal S$ to $L^\infty$ (see \cite{GSU20}). The definition above can be uniquely extended in such a way that $(-\Delta)^s$ acts as a continuous operator $(-\Delta)^s : H^r(\mathbb R^n)\rightarrow H^{r-2s}(\mathbb R^n)$ for all $r\in\mathbb R$. Further extensions are possible to Sobolev spaces of negative exponent and to $L^p$-based Sobolev spaces (\cite{CMR20, GSU20}). It is possible to give many other definitions of the fractional Laplacian, which at least in the case $s\in(0,1)$ can be shown to be equivalent to ours (\cite{Kw15}). In particular, the fractional Laplacian can be defined as a singular integral by
$$ (-\Delta)^su(x) := C_{n,s}\,PV\int_{\mathbb R^n}\frac{u(x)-u(y)}{|x-y|^{n+2s}}dy ,$$
where $C_{n,s}:= \frac{4^s \Gamma(n/2+s)}{\pi^{n/2}|\Gamma(-s)|}$.
\vspace{3mm}

The fractional Laplacian presents a nonlocal behaviour, as made evident by the following property it enjoys:

\begin{The}[UCP for the fractional Laplacian]
Let $s\in\mathbb R^+\setminus\mathbb Z$, and assume $u\in H^r(\mathbb R^n)$ for some $r\in\mathbb R$. If $u=(-\Delta)^su=0$ in a non-empty open set $V$, then $u\equiv 0$ in $\mathbb R^n$.
\end{The}

The proof of this theorem can be found in \cite{CMR20}, where the authors explore also the case of negative exponent $s$. The main case $s\in (0,1)$ upon which the proof is based was proved in \cite{GSU20}. One more property of the fractional Laplacian we shall use is the following fractional Poincar\'e inequality. 

\begin{The}
\label{thm:generalpoincare}
Let $\dimens \geq 1$, $s\geq t\geq 0$, $K\subset\R^\dimens$ a compact set and $u\in H_K^s(\R^\dimens)$. There exists a constant $\widetilde{c}=\widetilde{c}(n, K, s)> 0$ such that
\begin{equation*}
\aabs{(-\Delta)^{t/2}u}_{L^2(\R^\dimens)}\leq \widetilde{c}\aabs{(-\Delta)^{s/2}u}_{L^2(\R^\dimens)}.
\end{equation*}
\end{The}
\noindent Many variously flavoured proofs of the above statement can be found in \cite{CMR20}. 
\\

Next, we shall define nonlocal counterparts to the gradient and divergence operators as in \cite{C20, CMR20}. These are special instances of the general nonlocal vector calculus operators introduced in \cite{DGLZ12, DGLZ13}. For $u\in C^\infty_c(\mathbb R^n, \mathbb R^n)$ and $s\in (0,1)$ the fractional gradient is 
$$ \nabla^s u (x,y) := (u(y)-u(x))\otimes \zeta(x,y), \quad\mbox{ where } \zeta(x,y):=\frac{C_{n,s}^{1/2}}{\sqrt{2}} \frac{x-y}{|x-y|^{n/2+s+1}}$$
and since one sees that $\|\nabla^s u\|_{L^2(\mathbb R^{2n})}\leq \|u\|_{H^s(\mathbb R^n)}$, the definition is extended to act as $H^s(\mathbb R^n)\rightarrow L^2(\mathbb R^{2n})$ by density. Many more properties of the fractional gradient can be found in \cite{C20a}. The fractional divergence $(\nabla\cdot)^s$ is defined as the adjoint of $\nabla^s$, that is $$\langle (\nabla\cdot)^s v, u \rangle := \langle v, \nabla^s u \rangle $$
for all $v\in L^2(\mathbb R^{2n}), u\in H^s(\mathbb R^n)$. Thus $(\nabla\cdot)^s: L^2(\mathbb R^{2n})\rightarrow H^{-s}(\mathbb R^n)$. It is useful to keep in mind that for more regular functions one has $$(\nabla\cdot)^s v (x) = \frac{C_{n,s}^{1/2}}{\sqrt{2}}\int_{\mathbb R^n} \frac{v(x,y)+v(y,x)}{|x-y|^{n/2+s+1}}\cdot (x-y) dy.$$

More generally, we define the fractional gradient of order $s\in\mathbb R^+\setminus\mathbb Z$ as $\nabla^{s}:=\nabla^{s-\floor{s}}\nabla^{\floor{s}}$, and a corresponding relation holds by definition for the fractional divergence. Most importantly, one sees that the property $$(\nabla\cdot)^s\nabla^s u = (-\Delta)^s u$$ holds in $H^s(\mathbb R^n)$. Finally, we define the $j$-th fractional derivative as
$$\partial^s_j u(x,y) := \nabla^su(x,y)\cdot e_j = \frac{C_{n,s}^{1/2}}{\sqrt{2}} \frac{u(y)-u(x)}{|x-y|^{n/2+s+1}} (x_j-y_j).$$
In particular this definition implies that $\partial^s_ju \cdot \partial^s_i v = \partial^s_iu \cdot \partial^s_j v $ holds for all vectors $u,v$ and all $i,j\in\{1,...,n\}$. This property, which of course does not hold in the classical case $s=1$, will be fundamental for our arguments.

\addtocontents{toc}{\SkipTocEntry}
\subsection{H\"older spaces}

Let $r\in(0,1)$. Following \cite{Ta96}, we define the H\"older space $C^r(\mathbb R^n)$ as the set of bounded functions $u$ such that $$ |u(x)-u(y)|\leq C|x-y|^{r} $$
for all $x,y\in\mathbb R^n$. If $k\in\N$, we let $C^k(\mathbb R^n)$ be the set of all bounded continuous functions $u$ such that $D^\beta u$ is bounded and continuous for all multi-indexes $|\beta|\leq k$. Finally, if $r\in\R^+\setminus\Z^+$ we define $C^r(\mathbb R^n)$ to be the set of all functions $u\in C^{\floor{r}}(\mathbb R^n)$ with $D^\beta u \in C^{r-\floor{r}}(\mathbb R^n)$ for all multi-indexes $|\beta|=\floor{r}$.
\\

Next, we shall list some properties of H\"older spaces which will be useful in our arguments. It follows immediately from the definition that for all $r,r'\in\R^+$ one has the embedding
\begin{equation}\label{embedding}
    r\leq r' \Rightarrow C^{r'}(\mathbb R^n)\subseteq C^r(\mathbb R^n).
\end{equation}
The H\"older space $C^r(\R^n)$ is clearly closed under composition with smooth functions, that is
\begin{equation}\label{smooth} u \in C^r(\R^n),\, F\in C^\infty(\R) \Rightarrow F(u) \in C^r(\R^n). \end{equation}
It is also known that $C^r(\R^n)$ is an algebra for all $r\in\mathbb R^+\setminus\Z^+$, which means
\begin{equation}\label{algebra} u,v \in C^r(\R^n) \Rightarrow uv \in C^r(\R^n). \end{equation}
Let $m\in\mathbb R$ and assume $\Psi\in OPS^m_{1,0}$, i.e. $\Psi$ is a pseudodifferential operator with symbol in the H\"ormander class $S^m_{1,0}$. If both $r, r-m \in \R^+\setminus\Z^+$, then 
\begin{equation}\label{mapping}
    \Psi : C^r(\R^n) \rightarrow C^{r-m}(\R^n).
\end{equation}

For later purposes, we also show that the operator $(-\Delta)^{s-1} \partial_i \partial_j$ maps $C^r(\R^n)$ to $C^{r-2s}(\R^n)$. We write the symbol as 
\[
\abs{\xi}^{2s-2} \xi_i \xi_j = \psi(\xi) \abs{\xi}^{2s-2} \xi_i \xi_j + (1-\psi(\xi)) \abs{\xi}^{2s-2} \xi_i \xi_j
\]
where $\psi \in C^{\infty}_c(\mathbb R^n)$ satisfies $\psi = 1$ near $0$, and note that the second symbol has the right mapping properties by \eqref{mapping}. The first symbol gives rise to a convolution operator $u \mapsto k * u$, where 
$$
k = F^{-1}(\psi) * F^{-1}(\abs{\xi}^{2s-2} \xi_i \xi_j).
$$
Now the first function in the convolution is Schwartz and the second one is homogeneous of order $-n-2s$ and smooth outside of $0$ \cite{HO:analysis-of-pdos}, which proves that $k \in L^1(\mathbb R^n)$. Since $\psi(\xi)$ is compactly supported, the Fourier characterization of H\"older spaces \cite{TRI-interpolation-function-spaces} implies that $u \mapsto k * u$ is bounded between any two H\"older spaces. This proves that whenever $r, r-2s \in \mathbb R^+ \setminus \mathbb Z^+$, one has 
\begin{equation} \label{mapping2}
(-\Delta)^{s-1} \partial_i \partial_j: C^r(\R^n) \to C^{r-2s}(\R^n).    
\end{equation}

\addtocontents{toc}{\SkipTocEntry}
\subsection{Fundamentals of the classical theory of linear elasticity}\label{subsec-elast}
This section is a brief introduction to some of the fundamental concepts in the theory of linear elasticity. Our main reference in this respect is Landau's book \cite{landau-elasticity}; however, we do not necessarily restrict our discussion to the case $n=3$. 
\vspace{2mm}

The theory of elasticity studies the mechanics of the deformations of continuous media. Let $\Omega\subset\mathbb R^n$ be an open, bounded set representing a physical body. When forces are applied to $\Omega$, the body answers to that by changing shape and volume, i.e. each point $x\in \Omega$ is transferred to a new location $x'\in\mathbb R^n$. Thus it is possible to define a \emph{displacement vector field} $u$ such that $u(x)= x'-x$ for all $x\in\Omega$. 

\vspace{2mm}

By computing the change in distance between two points which were originally close to each other, one sees that the new infinitesimal distance $dl'$ is related to the original one $dl$ by
$$ dl'^2 - dl^2 = 2u_{ik}dx_idx_k ,$$
where $$ u_{ik}:= \frac{1}{2}\left( \frac{\partial u_i}{\partial x_k} + \frac{\partial u_k}{\partial x_i} + \frac{\partial u_l}{\partial x_i}\frac{\partial u_l}{\partial x_k} \right)$$
is called the \emph{strain tensor}. It is customary in the case of small deformations to neglect the second order terms of the strain tensor and use instead what is known as \emph{Cauchy's strain tensor}, \emph{linear strain tensor} or \emph{small strain tensor}:
$$\varepsilon_{ik} := \frac{1}{2}\left( \frac{\partial u_i}{\partial x_k} + \frac{\partial u_k}{\partial x_i} \right) \approx u_{ik}.$$
One immediately sees that the tensor $\varepsilon$ is symmetric.
\vspace{2mm}

When deformed, the body leaves its equilibrium state and some \emph{internal forces} (or \emph{stresses}) $F$ are generated, which attempt to return the body to its original undeformed state. The resultant of such forces acting on a region $\Omega'\subset\Omega$ can be computed as $\int_{\Omega'}FdV$. Assuming that such forces do not act on a distance, one should be able to express their resultant as an integral over $\partial\Omega'$. This suggests that $F$ should be of the form $\nabla\cdot\sigma$, where the new tensor $\sigma$ is called \emph{stress tensor}. Further investigation of the total moment of the forces acting on $\Omega'$ reveals that the stress tensor $\beta$ should also be symmetric.
\vspace{3mm}

Because we are interested in a theory of \emph{linear} elasticity, we shall assume a linear relationship between the stress tensor $\sigma$, representing the forces acting on $\Omega$, and the strain tensor $\varepsilon$, which represents the resulting deformation of the body. This gives rise to the following \emph{Hooke's law} (or \emph{constitutive equation of linear elasticity})

$$ \sigma_{ij} = C_{ijlm}\varepsilon_{lm}, $$

\noindent where the new fourth-order tensor $C$, which completely describes the elastic behaviour of the body, is called \emph{elasticity} or \emph{stiffness tensor}. Using \emph{Newton's second law}, we can eventually define the \emph{operator of classical elasticity} $\mathbf E$ as \begin{equation}\label{eq:classical} \mathbf E u := \nabla\cdot (C(\nabla u + \nabla u^T)).\end{equation}
One can also associate a \emph{potential energy} $U$ to $\mathbf E$:\begin{align*} U(u)&:= \frac{1}{2} \int_{\mathbb R^{n}} C_{ijlm}(x) \varepsilon_{ij}(x)\varepsilon_{lm}(x)\,dx .\end{align*}
This quantity is often assumed to be non-negative, with $U(u)=0$ holding if and only if $u= 0$.
\vspace{3mm}

Let us now assume that the material is \emph{isotropic}, i.e. it is completely characterized by properties which are independent of direction. In this case the elasticity tensor can be expressed as
$$ C_{ijlm} = \lambda \delta_{ij}\delta_{lm} + \mu( \delta_{il}\delta_{jm} + \delta_{im}\delta_{jl}), $$
where $\lambda, \mu$ are called \emph{Lam\'e parameters}. Consequently, Hooke's law becomes
$$ \sigma_{ij} = \lambda tr(\varepsilon)\delta_{ij} + 2\mu \varepsilon_{ij} = \left( \lambda+\frac{2\mu}{n} \right)tr(\varepsilon)\delta_{ij} + 2\mu \left( \varepsilon_{ij} - \frac{tr(\varepsilon)}{n}\delta_{ij} \right). $$
The quantities $k:=\lambda+2\mu/n$ and $\mu$ (this last one sometimes indicated by $G$) are respectively called \emph{bulk} and \emph{shear moduli}. Thermodynamic considerations ensure the positivity of both $k$ and $\mu$. 

The corresponding tensors $tr(\varepsilon)\mbox{Id}$ and $\varepsilon - \frac{tr(\varepsilon)}{n}\mbox{Id}$ are called \emph{dilational} and \emph{deviatoric strain tensors}. The first one of them represents the \emph{hydrostatic compression} of the body, i.e. a deformation in scale but not in shape, and it is independent of the coordinate system. On the other hand, the second tensor represents a \emph{pure shear}, i.e. a deformation in which the volume is unchanged, and it is trace-free.
\vspace{1mm}

It is also possible to compute the strain tensor given the stress tensor. By Hooke's law we have 
$$ tr(\sigma) = (n\lambda+2\mu) tr(\varepsilon) = nktr(\varepsilon),$$ 
and since $k>0$ we can write $\sigma_{ij} = \frac{\lambda tr(\sigma)}{nK}\delta_{ij} + 2\mu \varepsilon_{ij}$. This eventually gives
$$ \varepsilon_{ij} = \frac{1}{2\mu}\left(\sigma_{ij} -\frac{\lambda tr(\sigma)}{nk}\delta_{ij}\right) = s_{ijlm}\sigma_{lm}, $$
where the new tensor $s_{ijlm}:= \frac{\delta_{il}\delta_{jm}}{2\mu} - \frac{\lambda\delta_{ij}\delta_{lm}}{2n\mu k}$ is the \emph{isotropic compliance tensor}.

 The Poisson effect indicates the physical phenomenon observed in the study of elastic materials in which a body reacts to a compression (resp. extension) along one axis with an extension (resp. compression) in the perpendicular directions. The \emph{Poisson ratio} $\nu$ is defined as the amount of transverse extension divided by the amount of axial compression. It is easily computed in the case of a \emph{homogeneous deformation} of a thin rod along its axis. If such axis is oriented in the $e_n$ direction and the applied pressure is $p$, then $ \sigma_{ii}= p\delta_{in}$ for all $i\in\{1,...,n\}$. Therefore
$$ \varepsilon_{ii} = \frac{1}{2\mu}\left(\sigma_{ii} -\frac{\lambda tr(\sigma)}{nk}\right) =  \frac{p}{2\mu}\left(\delta_{in} -\frac{\lambda}{nk}\right), \quad \mbox{ for all } i\in\{1,...,n\}, $$
which gives
$$ \nu := -\frac{\varepsilon_{11}}{\varepsilon_{nn}} = -\frac{\delta_{1n} -\frac{\lambda}{nk}}{\delta_{nn} -\frac{\lambda}{nk}} = \frac{\lambda}{nk -\lambda} =  \frac{k-\frac{2\mu}{n}}{(n-1)k+\frac{2\mu}{n}} =  \frac{\frac{k}{\mu}-\frac{2}{n}}{(n-1)\frac{k}{\mu}+\frac{2}{n}} .$$
Given the positivity of $k$ and $\mu$, it is always true that $\nu\in\left( -1,\frac{1}{n-1} \right)$. It is also clear that $\nu$ depends only on the ratio $k/\mu$ rather than on the two moduli taken separately.

\section{The isotropic fractional elasticity operator}\label{sec-operator}
In this section we introduce a model for linear fractional elasticity derived from the classical one, which is related to the model proposed by Tarasov and Aifantis in \cite{TarasovAifantis-2018}. To this end we use the fractional divergence and gradient, and the result will be a new self-adjoint operator $\mathbf E^s$. We also assume throughout the paper the stiffness tensor $C$ to be isotropic and such that the associated Lam\'e parameters $L, M$ satisfy
\begin{enumerate}[label=(A\arabic*)]
    \item there exists $\eps>0$ such that $L,M \in C^{2s+\eps}(\mathbb R^n)$, 
    \item there exist $L_0, M_0 \in \mathbb R$ such that $L-L_0 = M-M_0 = 0$ in $\Omega_e$, and
    \item the functions $M$ and $K:=L+\frac{2M}{n}$ are positive and bounded away from $0$.
\end{enumerate}

Given that our nonlocal operators act on two-point functions, we need a preparatory Lemma.

\begin{Lem}[Square root of stiffness tensor]\label{square-root}
Let the Lam\'e parameters $L,M$ of the isotropic stiffness tensor $C$ be as in assumptions (A1)-(A3). There exists a unique pair of real valued functions $\lambda, \mu $ verifying (A1)-(A3) and $$ C^{1/2} : C^{1/2} = C,$$
where the new tensor $C^{1/2}$ is defined as $$ C^{1/2}_{ijlm} := \lambda \delta_{ij}\delta_{lm} + \mu( \delta_{il}\delta_{jm} + \delta_{im}\delta_{jl}).$$ Moreover, we have
\begin{equation}\label{two-dots}\begin{aligned}
    C^{1/2}_{ij\alpha\beta}(x)C^{1/2}_{\alpha\beta lm}(y) &= (n\lambda(x)\lambda(y)+2\lambda(x)\mu(y) + 2\lambda(y)\mu(x)) \delta_{ij}\delta_{lm} \\ & \quad + 2\mu(x)\mu(y)\delta_{il}\delta_{jm} + 2\mu(x)\mu(y)\delta_{im}\delta_{jl}
\end{aligned},\end{equation}
and therefore $$C^{1/2}_{ij\alpha\beta}(x)C^{1/2}_{\alpha\beta lm}(y) = C^{1/2}_{lm\alpha\beta}(y)C^{1/2}_{\alpha\beta ij}(x). $$
\end{Lem}

\begin{proof} Note that
\begin{align*}
    C^{1/2}_{ij\alpha\beta}C^{1/2}_{\alpha\beta lm}&= (\lambda\delta_{ij}\delta_{\alpha\beta} + \mu\delta_{i\alpha}\delta_{j\beta} + \mu\delta_{i\beta}\delta_{j\alpha})(\lambda\delta_{\alpha\beta}\delta_{lm} + \mu\delta_{\alpha l}\delta_{\beta m} + \mu\delta_{\alpha m}\delta_{\beta l}) \\ &= (n\lambda^2+4\lambda\mu) \delta_{ij}\delta_{lm} + 2\mu^2\delta_{il}\delta_{jm} + 2\mu^2\delta_{im}\delta_{jl} 
\end{align*} and
$$C_{ijlm} = L\delta_{ij}\delta_{lm} + M\delta_{il}\delta_{jm} + M\delta_{im}\delta_{jl}. $$ We look for $\lambda, \mu$ such that $C^{1/2}_{ij\alpha\beta}C^{1/2}_{\alpha\beta lm} = C_{ijlm}$.
Since $\mu$ must be positive, we have $\mu := \sqrt{M/2}$, which also ensures that $\mu$ is bounded away from $0$. This gives two possible choices for $\lambda$, namely $\lambda_\pm = \frac{1}{n}(\pm\sqrt{2M+nL}-\sqrt{2M})$. However, the required positivity of the coefficient $k:=\lambda+2\mu/n$ ensures that $\lambda = \frac{1}{n}(\sqrt{2M+nL}-\sqrt{2M})$. Now $k= \sqrt{K/n}$, and thus it is bounded away from $0$. This proves that $\lambda,\mu$ satisfy condition (A3). Given that the square root is a smooth function when considered far from $0$, by formula \eqref{smooth} we deduce $\mu, k \in C^{2s+\eps}(\mathbb R^n)$, which in turn implies $\lambda\in C^{2s+\eps}(\mathbb R^n)$ as well and proves (A1). Since $L,M$ are constant outside of $\Omega$, so must be $\lambda,\mu$ too, which proves (A2). The last equalities in the statement of the lemma follow easily from the computations above.
\end{proof}

Similarly to what was done in \cite{C20}, we can define the new fractional elasticity operator
\begin{equation}\label{eq:main-operator}
\mathbf E ^s u := (\nabla\cdot)^s(C^{1/2}(x):C^{1/2}(y)(\nabla^s u + \nabla^{s}u^T)(x,y)).    
\end{equation}
Observe that this corresponds to taking as fractional Cauchy's strain tensor
\begin{equation}\label{strain-tensor}
     \varepsilon_{ik}^s := \frac{1}{2}(\partial^s_ku_i + \partial^s_iu_k), \qquad \mbox{i.e.}\quad \varepsilon^s := \frac{1}{2}(\nabla^su+\nabla^su^T), 
\end{equation}
and then as a fractional, "symmetrized" version of Hooke's law
$$ \sigma^s_{ij}(x,y) = C^{1/2}_{ij\alpha\beta}(x)C^{1/2}_{\alpha\beta lm}(y)\varepsilon^s_{lm}(x,y). $$

We next prove the following lemma about $\mathbf E^s$, which motivates definition \eqref{eq:main-operator}:

\begin{Lem}[Properties of $\mathbf E^s$]\label{self-adjointness} Let $L,M$ satisfy assumptions (A1), (A2), and let $s\in(0,1)$. The operator $\mathbf E^s$ is self-adjoint, and it maps $H^s(\mathbb R^n)$ to $H^{-s}(\mathbb R^n)$.
\end{Lem}

\begin{proof} We start from the proof of the mapping property. By Lemma \ref{square-root} both $\lambda$ and $\mu$ belong to $C^{2s+\eps}(\R^n)\subset L^\infty(\mathbb R^n)$, which implies that the two-point functions $\lambda(x)\lambda(y), \mu(x)\mu(y)$ and $\lambda(x)\mu(y)$ all belong to $L^\infty(\mathbb R^{2n})$. The result then follows by the mapping properties of $\nabla^s$ and $(\nabla\cdot)^s$ from Section \ref{operators} and equation \eqref{two-dots}.

In order to see the self-adjointness, recall the minor and major symmetries of the stiffness tensor $C_{ij\alpha\beta}^{1/2}=C_{ji \alpha\beta}^{1/2}$, $C_{ij\alpha\beta}^{1/2}=C_{\alpha\beta ij}^{1/2}$, as well as the last equality from Lemma \ref{square-root}. Then we compute for $\phi\in H^{s}(\mathbb R^n)$
\begin{align*}
    \langle \mathbf E^s u, \phi \rangle & = \langle (\nabla\cdot)^s(C^{1/2}(x):C^{1/2}(y)(\nabla^s u + \nabla^{s}u^T)), \phi \rangle \\ & = \langle C^{1/2}(x):C^{1/2}(y)(\nabla^s u + \nabla^{s}u^T), \nabla^s\phi \rangle \\ & =  \int_{\mathbb R^{2n}} C_{ij\alpha\beta}^{1/2}(x)C^{1/2}_{\alpha\beta lm}(y)(\partial^s_m u_l + \partial^s_l u_m) \partial^s_j     \phi_i \,dxdy \\ & =  2\int_{\mathbb R^{2n}} C_{ij\alpha\beta}^{1/2}(x)C^{1/2}_{\alpha\beta lm}(y)\partial^s_m u_l \partial^s_j \phi_i \,dxdy\\ & = \int_{\mathbb R^{2n}} \partial^s_m u_l \, C_{ij\alpha\beta}^{1/2}(x)C^{1/2}_{\alpha\beta lm}(y) (\partial^s_j \phi_i + \partial^s_i \phi_j  )\,dxdy \\ & = \langle \nabla^s u, C^{1/2}(x) : C^{1/2}(y) (\nabla^s \phi + \nabla^s \phi^T) \rangle \\ & = \langle u, (\nabla\cdot)^s(C^{1/2}(x):C^{1/2}(y)(\nabla^s \phi + \nabla^{s}\phi^T)) \rangle = \langle u, \mathbf E^s \phi \rangle. 
\end{align*}
\end{proof}

\begin{Rem}
For $n=1$, $\mathbf E ^s$ reduces to the fractional conductivity operator. In fact in that case $C$ is a positive scalar function $\gamma:\mathbb R\rightarrow \mathbb R$, and $\nabla^s u$ can be written as
$$\nabla^s u(x,y) = \nabla^su^T(x,y) = \frac{\mathcal C_{1,s}^{1/2}}{\sqrt{2}} \frac{u(y)-u(x)}{|x-y|^{1/2+s}}\, \mbox{sgn}(x-y)$$
\end{Rem}

We also define the fractional potential energy $U^s$ as
\begin{align*} U^s(u)&:= \frac{1}{2} \int_{\mathbb R^{2n}} C^{1/2}_{ij\alpha\beta}(x)C^{1/2}_{\alpha\beta lm}(y) \varepsilon^s_{ij}(x,y)\varepsilon_{lm}^s(x,y)\,dydx, \end{align*}
and prove the following Lemma:

\begin{Lem}[Positive definiteness of $U^s$]\label{pot-energy}
Let $K\subset\mathbb R^n$ be a compact set. There exist two constants $c,C>0$ such that the inequality $c\|u\|_{H^s}\leq U^s(u)\leq C\|u\|^2_{H^s}$ holds for all $u\in H^s_K(\mathbb R^n)$. As a consequence, $U^s(u)=0$ if and only if $u=0$.
\end{Lem}

\begin{proof} 
We first claim that 
\begin{equation}\label{positivity} C^{1/2}_{ij\alpha\beta}(x)C^{1/2}_{\alpha\beta lm}(y)v_iw_jv_lw_m \gtrsim v_iw_j v_i w_j = |v|^2|w|^2 \end{equation}
holds for all vectors $v=v(x,y)$, $w=w(x,y)$ in $\mathbb R^n$. In fact, if \eqref{positivity} holds, then the definition of the fractional strain tensor \eqref{strain-tensor}, the symmetries of $C^{1/2}$, and the definition of $\nabla^s u$ imply that 
\begin{align*}
U^s(u) &= 2 \int_{\mathbb R^{2n}} C^{1/2}_{ij\alpha\beta}(x)C^{1/2}_{\alpha\beta lm}(y) (\nabla^s u)_{ij}(\nabla^s u)_{lm} dxdy \\
 &\gtrsim \langle \mathbb \nabla^s u, \nabla^s u\rangle_{L^2(\mathbb R^{2n})} =\langle (-\Delta)^s u, u \rangle_{L^2(\mathbb R^n)} = \|(-\Delta)^{s/2}u\|_{L^2(\mathbb R^n)}^2.
\end{align*}
Using the fractional Poincar\'e inequality from Theorem \ref{thm:generalpoincare} for $u \in H^s_K(\mathbb R^n)$, we finally get
$$ \|u\|^2_{H^s(\mathbb R^n)} \leq \|u\|^2_{L^2(\mathbb R^n)} +\|(-\Delta)^{s/2}u\|^2_{L^2(\mathbb R^n)} \lesssim \|(-\Delta)^{s/2}u\|^2_{L^2(\mathbb R^n)} \lesssim U^s(u).  $$
Thus for the lower bound it suffices to show \eqref{positivity}. Using Lemma \ref{square-root} we see that the left hand side of \eqref{positivity} can be written as
$$ \left(nk(x)k(y) + \frac{2(n-2)}{n}\mu(x)\mu(y)\right)(v\cdot w)^2 + 2\mu(x)\mu(y) |v|^2|w|^2, $$
where we recall our previous definition $k := \lambda + 2\mu/n$. Since $k(x), \mu(x)$ are known to be larger than a positive constant for all $x\in\mathbb R^n$ and also $n\geq 2$, the lower bound is proved. The reverse inequality $U^s(u) \lesssim \|u\|_{H^s}^2$ follows immediately from the boundedness of the coefficients $L,M$ (see also formula \eqref{boundedness}). This completes the proof of the lemma.
\end{proof}

While the definition of the fractional elasticity operator $\mathbf E^s$ is very useful in order to prove its self-adjointness, when it comes to studying other properties it is more convenient to rewrite it in a different way.
\\

\begin{Lem}[Reduction lemma]\label{reduction-lemma}
Let $s\in(0,1)$ and $u\in H^s(\mathbb R^n)$. Then in weak sense we have
\begin{align}\begin{split}\label{big-formula}
    (n/2+s)\mathbf E^su & = (2n+4s+n')\mu (-\Delta)^s(u\mu) + nk(-\Delta)^s(uk) \\ & \quad -2n's\mu (-\Delta)^{s-1}\nabla\nabla\cdot(u\mu) - 2ns k (-\Delta)^{s-1}\nabla\nabla\cdot(uk) \\ & \quad +2su\cdot \left\{ n'\mu (-\Delta)^{s-1}\nabla^2 \mu +nk (-\Delta)^{s-1}\nabla^2k \right\} \\ & \quad - u \left\{ (2n+4s+n')\mu(-\Delta)^s\mu + nk (-\Delta)^s k \right\},
\end{split}\end{align}
where $n':= 2(n-2)/n$.
\end{Lem}

\begin{proof} \textbf{Step 1.}
By the computations in Lemmas \ref{self-adjointness} and \ref{square-root}, we see that for all test functions $\phi\in H^s(\mathbb R^n)$
\begin{align*}
    \frac{1}{2}\langle \mathbf E^su,\phi \rangle & =  \langle C_{ij\alpha\beta}^{1/2}(x)C_{\alpha\beta lm}^{1/2}(y) \partial^s_l u_m, \partial^s_i \phi_j\rangle \\ & = \langle  (n\lambda(x)\lambda(y)+2\lambda(x)\mu(y) + 2\lambda(y)\mu(x)) \partial^s_j u_j , \partial^s_i \phi_i\rangle + \\ & \quad + 2\langle \mu(x)\mu(y)\partial^s_j u_i, \partial^s_j \phi_i\rangle + 2\langle\mu(x)\mu(y)\partial^s_i u_j, \partial^s_j \phi_i\rangle \\ & = \langle  (n\lambda(x)\lambda(y)+2\lambda(x)\mu(y) + 2\lambda(y)\mu(x)+2\mu(x)\mu(y)) \partial^s_i u_j , \partial^s_j \phi_i\rangle + \\ & \quad + 2\langle \mu(x)\mu(y)\partial^s_j u_i, \partial^s_j \phi_i\rangle.
\end{align*}
For the last equality we used the relation $\partial^s_j u_j\, \partial^s_i \phi_i = \partial^s_i u_j\, \partial^s_j \phi_i$, which follows directly from the definition of the fractional gradient. Using the definitions $k:= \lambda + 2\mu /n$ and $n':= 2(n-2)/n$, we can write
\begin{align*}
    \langle \mathbf E^su,\phi \rangle & = \langle 2(n k(x)k(y) + n'\mu(x)\mu(y))(\nabla^s u)^T + 4\mu(x)\mu(y)\nabla^s u, \nabla^s \phi \rangle \\ & =: \langle  a(x,y)\nabla^s u + b(x,y)(\nabla^s u)^T, \nabla^s \phi \rangle .
\end{align*}
By the definition of the fractional gradient we then have
\begin{align*}
    \langle \mathbf E^su,\phi \rangle & = \langle  a(x,y)(u(y)-u(x))\otimes \zeta + b(x,y)\zeta\otimes(u(y)-u(x)), (\phi(y)-\phi(x))\otimes \zeta \rangle \\ & = \langle  a(x,y)|\zeta|^2(u(y)-u(x))  + b(x,y)(\zeta\otimes(u(y)-u(x)))\cdot \zeta, \phi(y)-\phi(x)\rangle \\ & = \langle  a(x,y)|\zeta|^2(u(y)-u(x)), \phi(y)-\phi(x)\rangle  \\ & \quad + \langle b(x,y)(\zeta\otimes\zeta)\cdot(u(y)-u(x)), \phi(y)-\phi(x)\rangle,
\end{align*}
where at the second and third steps we used Lemma \ref{product-properties}. Since
$$\frac{x-y}{|x-y|^{m+2}}=\frac{\nabla_y(|x-y|^{-m})}{m}=-\frac{\nabla_x(|x-y|^{-m})}{m}$$ holds for all $m\neq 0$, and also the identity $\nabla\psi\otimes v = \nabla(\psi v)-\psi \nabla v$ holds for scalar $\psi$ and vector $v$ (see again Lemma \ref{product-properties}), we have
\begin{align*}
    \frac{(x-y)\otimes (x-y)}{|x-y|^{n+2s+2}} & = \frac{\nabla_y(|x-y|^{-(n+2s)})}{n+2s}\otimes (x-y) \\ & = \frac{1}{n+2s}\left( \frac{Id}{|x-y|^{n+2s}}+\nabla_y\left(\frac{x-y}{|x-y|^{n+2s}}\right) \right) \\ & = \frac{1}{n+2s}\left( \frac{Id}{|x-y|^{n+2s}}-\frac{\nabla_y\nabla_x (|x-y|^{-(n+2s-2)})}{n+2s-2} \right),
\end{align*}
which implies 
\begin{multline*}
    (\zeta\otimes\zeta)\cdot(u(y)-u(x)) = \frac{|\zeta|^2(u(y)-u(x))}{n+2s}
\\    
    -\frac{C_{n,s}\nabla_y\nabla_x (|x-y|^{-(n+2s-2)}) }{2(n+2s)(n+2s-2)}\cdot(u(y)-u(x)).
\end{multline*}
Using $C_{n,s}=2s(n+2s-2) C_{n,s-1}$ we get
\begin{align*}
    \langle \mathbf E^su,\phi \rangle & = \langle  \left(a(x,y)+\frac{b(x,y)}{n+2s}\right)|\zeta|^2(u(y)-u(x)), \phi(y)-\phi(x)\rangle  \\ & \quad - \frac{s\, C_{n,s-1}}{n+2s}\langle b(x,y)\nabla_y\nabla_x (|x-y|^{-(n+2s-2)})\cdot(u(y)-u(x)), \phi(y)-\phi(x)\rangle \\ & =: \frac{I_1 - sI_2}{n/2+s}.
\end{align*}

\textbf{Step 2.} For the first term $I_1$ we compute
$$(n+2s)a(x,y)+b(x,y) = 2(2n+4s+n')\mu(x)\mu(y)+2nk(x)k(y),$$
and then obtain
\begin{align*}
    I_1 & = (2n+4s+n')\langle \mu(x)\mu(y)\nabla^s u, \nabla^s \phi \rangle + n \langle k(x)k(y)\nabla^s u, \nabla^s \phi \rangle \\ & =  (2n+4s+n') \langle\mathbf  C^s_{\mu^2} u, \phi\rangle  + n \langle\mathbf C^s_{k^2} u,\phi\rangle,
\end{align*}
where $\mathbf C^s_{\mu^2},\mathbf C^s_{k^2}$ are fractional conductivity operators, as studied in \cite{C20}. Since by assumptions (A1)-(A3) $\mu^2, k^2$ are conductivities in the sense of \cite{C20}, Theorem 3.1 from this paper can be applied. This leads to 
\begin{equation}\label{first-half}\begin{split}
    I_1  & = (2n+4s+n') \langle\mu (-\Delta)^s(u\mu) - \mu u (-\Delta)^s({\mu-\mu_0}), \phi\rangle \\ & \quad +n \langle k (-\Delta)^s(uk) - k u (-\Delta)^s({k-k_0}),\phi\rangle. 
\end{split}\end{equation}

For the second term $I_2$ we want to integrate by parts twice. For the sake of simplicity, we will show our computations only for the term with $k$ coming from $b$ (the term with $\mu$ is treated in the same way). For $w := ku$ we have the integral
\begin{align*}
    \int_{\R^{2n}} & \nabla_y\nabla_x (|x-y|^{-(n+2s-2)}):[(\phi(y)-\phi(x)) \otimes (k(x)w(y)-k(y)w(x))]dydx \\ & = -\int_{\R^{2n}} \nabla_x (|x-y|^{-(n+2s-2)})\cdot \left\{ (\nabla\cdot\phi)(y)(k(x)w(y)-k(y)w(x)) \right. \\ & \qquad \qquad +\left. (\phi(y)-\phi(x))\cdot (k(x)\nabla w(y)-\nabla k(y) \otimes w(x)) \right\}dydx \\ & = \int_{\R^{2n}} \frac{1}{|x-y|^{n+2s-2}} (\nabla\cdot)_x \left\{ (\nabla\cdot\phi)(y)(k(x)w(y)-k(y)w(x)) \right. \\ & \qquad \qquad +\left. (\phi(y)-\phi(x))\cdot (k(x)\nabla w(y)-\nabla k(y) \otimes w(x)) \right\}dydx  \\ & = \int_{\R^{2n}} \frac{1}{|x-y|^{n+2s-2}}  \left\{ (\nabla k(x)\cdot w(y)-k(y)(\nabla\cdot w)(x))(\nabla\cdot\phi)(y) \right. \\ & \qquad \qquad   +(w(x)\otimes \nabla k(y)-k(x)(\nabla w(y))^T ):\nabla\phi(x) \\ & \qquad \qquad \left. + ( \nabla w(y)\cdot \nabla k(x) - \nabla k(y) (\nabla\cdot w)(x) )\cdot(\phi(y)-\phi(x))  \right\}dydx \\ & = \int_{\R^{2n}} \frac{1}{|x-y|^{n+2s-2}}  \left\{ (\nabla k(x)\cdot w(y)-k(y)(\nabla\cdot w)(x))(\nabla\cdot\phi)(y) \right. \\ & \qquad \qquad   +(w(y)\otimes \nabla k(x)-k(y)(\nabla w(x))^T ):\nabla\phi(y) \\ & \qquad \qquad \left. + [ \nabla w(y)\cdot \nabla k(x) - \nabla k(y) (\nabla\cdot w)(x) \right.\\ & \qquad \qquad \quad    \left.- \nabla w(x)\cdot \nabla k(y) + \nabla k(x) (\nabla\cdot w)(y) ]\cdot\phi(y)  \right\}dydx,
\end{align*}
where at the last step we exchanged the $x,y$ variables in the last two terms. Observe that the $n+2s-2 < n$ exponent ensures that all the above integrals are well-defined. If we define the new operator $R:= \frac{(-\Delta)^{s-1}}{C_{n,s-1}}$, we can rewrite the last line as
\begin{align*}
    \int_{\R^{n}} &  \{R\nabla k\cdot w-kR(\nabla\cdot w)\}(\nabla\cdot\phi)   +\{w\otimes R\nabla k-kR(\nabla w)^T \}:\nabla\phi \\ & + \{ \nabla w\cdot R\nabla k - \nabla k R(\nabla\cdot w)  - R\nabla w\cdot \nabla k + R\nabla k (\nabla\cdot w) \}\cdot\phi \,dy.
\end{align*}
We integrate by parts in the first two terms one more time. Since $R$ commutes with the derivatives, we get many cancellations, and eventually
\begin{align*}
    & \int_{\R^{n}} -\nabla\{R\nabla k\cdot w-kR(\nabla\cdot w)\}\cdot\phi   -\nabla\cdot\{w\otimes R\nabla k-kR(\nabla w)^T \}\cdot\phi \\ & \qquad + \{ \nabla w\cdot R\nabla k - \nabla k R(\nabla\cdot w)  - R\nabla w\cdot \nabla k + R\nabla k (\nabla\cdot w) \}\cdot\phi \,dy \\ & = 2\langle kR\nabla\nabla\cdot w - w\cdot R\nabla^2 k ,\phi\rangle \\ & = 2\langle kR\nabla\nabla\cdot (ku)- ku\cdot R\nabla^2 k,\phi\rangle.
\end{align*}
Coming back to $I_2$, we have obtained
\begin{equation}\label{second-half}
    \begin{split}
       I_2 & =  2n\langle  k(-\Delta)^{s-1}\nabla\nabla\cdot (ku)-ku\cdot (-\Delta)^{s-1}\nabla^2 k,\phi\rangle \\ & \quad +2n'\langle  \mu (-\Delta)^{s-1}\nabla\nabla\cdot (\mu u)-\mu u\cdot (-\Delta)^{s-1}\nabla^2 \mu,\phi\rangle,
    \end{split}
\end{equation}
and the wanted formula finally follows as a combination of  \eqref{first-half} and \eqref{second-half}; however, we still need to make sure that all the involved terms make sense in $H^{-s}(\R^n)$. Recall that by assumptions (A1), (A2) there exist $\mu_0, k_0 > 0$ such that $\tilde\mu :=\mu-\mu_0 \in C^{2s+\eps}_c(\Omega)$ and $\tilde k := k-k_0 \in C^{2s+\eps}_c(\Omega)$. For each term we can use a decomposition of the kind
$$ \mu(-\Delta)^s(u\mu) = \tilde\mu(-\Delta)^s(u\tilde\mu) + \mu_0\tilde\mu(-\Delta)^su + \mu_0(-\Delta)^s(u\tilde\mu) + \mu_0^2(-\Delta)^su,$$
and thus it suffices to study the terms in \eqref{first-half} and \eqref{second-half} with $\mu, k$ substituted by $\tilde\mu, \tilde k$.

Given that $\tilde{\mu}, \tilde k \in C^{2s+\epsilon}(\R^n)$, they both map $H^s(\R^n)$ to itself by \cite{TRI-interpolation-function-spaces}, Section 3.3.2. Since the space of multipliers on $H^s(\mathbb R^n)$ coincides with that of the multipliers on $H^{-s}(\R^n)$ (see \cite{MS-theory-of-sobolev-multipliers, CMRU20}), we have $\mu (-\Delta)^s(u\mu)\in H^{-s}(\mathbb R^n)$. This same reasoning shows that the first two terms on the right hand sides of  \eqref{first-half} and \eqref{second-half} all make sense in $H^{-s}(\mathbb R^n)$.    

 For the remaining parts of \eqref{first-half} and \eqref{second-half} we can proceed as follows, taking as an example the term $u\cdot\tilde\mu(-\Delta)^{s-1}\nabla^2\tilde\mu$. It suffices to show $\tilde\mu(-\Delta)^{s-1}\nabla^2\tilde\mu \in L^\infty(\mathbb R^n)$, since this is a set of multipliers on $L^2(\mathbb R^n)$. Moreover, given that $\tilde\mu \in C^{2s+\eps}_c(\Omega)\subset L^\infty(\R^n)$, it is enough to show that the operator $(-\Delta)^{s-1}\nabla^2$ maps $C^{2s+\eps}_c(\Omega)$ to $L^\infty(\mathbb R^n)$. However, by formula \eqref{mapping2} we have $(-\Delta)^{s-1}\nabla^2:C^{2s+\eps}(\R^n) \rightarrow C^{\eps}(\R^n)$, so that the wanted mapping property follows from $C^{\eps}(\R^n)\subset L^\infty(\R^n)$. This concludes the proof.
\end{proof}

\begin{Rem}\label{matrix-u}
The above reduction can similarly be performed if $u\in H^s(\R^n)$ is matrix-valued instead than vector-valued. In this case the fractional elasticity operator $\mathbf E^s$ is weakly defined as
$$ \langle \mathbf E^su,\phi \rangle =2  \langle C_{ij\alpha\beta}^{1/2}(x)C_{\alpha\beta lm}^{1/2}(y) \partial^s_l u_{mp}, \partial^s_i \phi_{jp}\rangle, $$
for all matrix-valued test functions $\phi\in H^s(\R^n)$. The first step of the proof, which only deals with the vector $\zeta$, and formula \eqref{first-half} are unchanged in the matrix case, apart from the additional component indicated by the index $p$. For the term $I_2$ in this case we compute the integral
\begin{align*}
    \int_{\R^{2n}}  \partial_{y,i}\partial_{x,j}& (|x-y|^{-(n+2s-2)}) [(\phi_{il}(y)-\phi_{il}(x)) (k(x)w_{jl}(y)-k(y)w_{jl}(x))]dydx \\ & = 2\int_{\R^n}\phi : \left( kR\nabla\nabla\cdot w - R\nabla^2k\cdot w \right) dy \end{align*}
following the same integration by parts technique shown in the second step of the proof of Lemma \ref{reduction-lemma}. If in particular there exists a scalar function $r$ such that $u=r Id$, and thus $w$ commutes with all matrices, we obtain that \eqref{big-formula} holds. This observation will be used in the last steps of the proof of the main theorem.
\end{Rem} 

\section{Well-posedness and the DN map}\label{sec-wellposedness}

We begin this section by defining our problem of interest. Let $s\in (0,1)$, and assume $\Omega\subset\mathbb R^n$ is a bounded open set. In the \emph{direct problem for the isotropic fractional elasticity equation} we are given an exterior value $f\in H^s(\mathbb R^n)$, and we want to find a weak solution $u$ to 
\begin{equation}\label{original-problem}
    \begin{array}{rll}
       \mathbf E^su & =0 & \quad \mbox{ in } \Omega \\
       u & = f & \quad \mbox{ in } \Omega_e
    \end{array}.
\end{equation}
Here the condition $u=f$ in $\Omega_e$ should be intended in the sense that $u-f\in \widetilde H^s(\Omega)$. In order to define what we mean by a weak solution, we introduce the following bilinear form. Using the definition of the operator, we write for $u,v\in C^\infty_c(\mathbb R^n)$
\begin{align*}
    B_{L,M}(u,v) & := \langle C^{1/2}(x):C^{1/2}(y) \nabla^s u, \nabla^s v \rangle. 
\end{align*} 

\noindent It is immediately seen that $B_{L,M}$ is symmetric. Boundedness in $H^s(\mathbb R^n)\times H^s(\mathbb R^n)$ follows easily thanks to the assumption $L,M\in L^\infty(\mathbb R^n)$, which implies $\lambda,\mu\in L^\infty(\mathbb R^n)$ as well:
\begin{align}\label{boundedness}\begin{split} 
    |B_{L,M}(u,v)| & \leq \|C^{1/2}(x):C^{1/2}(y)\nabla^s u\|_{L^2(\mathbb R^{2n})} \|\nabla^s v\|_{L^2(\mathbb R^{2n})} \\ & \lesssim \|\nabla^s u\|_{L^2(\mathbb R^{2n})}\|\nabla^s v\|_{L^2(\mathbb R^{2n})} \leq \|u\|_{H^s(\mathbb R^{n})}\|v\|_{H^s(\mathbb R^{n})}.
\end{split}\end{align} 
With this, we can extend the definition of $B_{L,M}$ to act on $H^s(\R^n)\times H^s(\R^n)$ by density. We can now say that 
$$u\in H^s(\mathbb R^n) \mbox{ is a weak solution to \eqref{original-problem} if and only if } B_{L,M}(u,\phi)=0 \mbox{ for all } \phi\in\widetilde H^s(\Omega),$$ 
and $u-f\in\widetilde H^s(\Omega)$.
More generally, we say that $u\in H^s(\mathbb R^n)$ is a weak solution to the inhomogeneous problem 
\begin{equation}\label{inhomo-problem}
    \begin{array}{rll}
       \mathbf E^su & = F & \quad \mbox{ in } \Omega \\
       u & = f & \quad \mbox{ in } \Omega_e
    \end{array},
\end{equation}
where $F\in H^{-s}(\Omega)$, if and only if $B_{L,M}(u,\phi)=F(\phi)$ holds for all $\phi\in\widetilde H^s(\Omega)$, and $u-f\in\widetilde H^s(\Omega)$.
\\ 

For our problem \eqref{inhomo-problem} we have the following well-posedness result:

\begin{Pro}[Well-posedness]\label{well-posedness}
Let $s \in (0,1)$ and assume $\Omega \subset \mathbb R^n$ is a bounded open set. For any $f\in H^s(\mathbb R^n)$ and $F\in H^{-s}(\Omega)$ there exists a unique $u\in H^s(\mathbb R^n)$ such that $u-f \in \widetilde H^s(\Omega)$ and
$$ B_{L,M}(u,\phi) = F(\phi) \quad \mbox{for all} \quad \phi \in \widetilde H^s(\Omega).$$
Moreover, the following estimate holds: 
$$ \|u\|_{H^s(\mathbb R^n)} \leq C\left( \|f\|_{H^s(\mathbb R^n)} + \|F\|_{H^{-s}(\Omega)} \right). $$
\end{Pro}

\begin{proof}
By letting $\tilde u := u-f$, we can reduce the above problem to the one of finding a unique $\tilde u \in \widetilde H^s(\Omega)$ such that $B_{L,M}(\tilde u, \phi)  = \tilde F(\phi)$, where $\tilde F := F - B_{L,M}(f,\cdot) $ belongs to $(\widetilde H^s(\Omega))^*$ because of the boundedness estimate \eqref{boundedness} for the bilinear form:
$$|\tilde F(\phi)| \leq |F(\phi)| + |B_{L,M}(f,\phi)|  \leq (\aabs{F}_{H^{-s}(\Omega)}+c\|f\|_{H^s(\mathbb R^n)})\|\phi\|_{H^s(\mathbb R^n)}.$$

Observe that $B_{L,M}(\cdot,\cdot)$ gives an equivalent inner product on $\widetilde H^s(\Omega)$, because by Lemma \ref{pot-energy} the fractional potential energy $U^s(v)$ always verifies $\|v\|^2_{H^s}\lesssim U^s(v)$ , and also it vanishes if and only if $v=0$. In fact, 
\begin{align*} B_{L,M}(v,v) & = \langle\nabla^s v, C^{1/2}(x):C^{1/2}(y) \nabla^s v \rangle \\ & = \int_{\mathbb R^{2n}} C^{1/2}_{ij\alpha\beta}(x)C^{1/2}_{\alpha\beta lm}(y) \varepsilon^s_{ij}(x,y)\varepsilon_{lm}^s(x,y)\,dydx  = 2\, U^s(v).  \end{align*}
The Riesz representation theorem now ensures the existence of a bounded linear operator $G: {H}^{-s}(\Omega) \rightarrow \widetilde H^s(\Omega)$ associating each functional in ${H}^{-s}(\Omega)$ to its unique representative in the inner product given by $B_{L,M}(\cdot,\cdot)$ on $\widetilde H^s(\Omega)$. Thus the wanted (unique) solution $\tilde u\in \widetilde H^s(\Omega)$ can be defined as $\tilde u := G \tilde F$, and it verifies $$ B_{L,M}(\tilde u,\phi) = \tilde F(\phi)\quad \mbox{for all}\quad \phi\in \widetilde H^s(\Omega). $$
 The boundedness of $G$ and the definition of $\tilde u$ eventually give the estimate
 \begin{align*}
     \|u\|_{H^s(\mathbb R^n)} & \leq  \|f\|_{H^s(\mathbb R^n)} + \|\tilde u\|_{H^s(\Omega)}  =  \|f\|_{H^s(\mathbb R^n)} + \|G\tilde F\|_{H^s(\Omega)} \\ & \leq \|f\|_{H^s(\mathbb R^n)} + C\|\tilde F\|_{H^{-s}(\Omega)} \leq C\left( \|f\|_{H^s(\mathbb R^n)} + \|F\|_{H^{-s}(\Omega)} \right).
 \end{align*}
 \end{proof}

In light of Proposition \ref{well-posedness}, we can define a Poisson operator $P_{L,M}$ of $H^s(\mathbb R^n)$ into itself: if $f\in H^s(\mathbb R^n)$ is any exterior datum, then $P_{L,M}f$ is by definition the unique solution to the homogeneous problem \eqref{original-problem}. It of course follows from Proposition \ref{well-posedness} that $P_{L,M}$ is a bounded operator.

\begin{Rem}\label{that-one-property}
Let $f,g$ be exterior values, let $u_f:=P_{L,M}f$ and $u_g:=P_{L,M}g$ be the unique solutions corresponding to them, and let $e_f, e_g$ be any extensions. Then
\begin{align*}
    B_{L,M}(u_f, e_g) & = \langle C^{1/2}(x):C^{1/2}(y)\nabla^s u_f, \nabla^s e_g \rangle = \langle \mathbf E^s u_f, e_g\rangle = \langle \mathbf E^s u_f, e_g\rangle_{\Omega_e} \\ & =  \langle \mathbf E^s u_f, u_g\rangle_{\Omega_e} = \langle \mathbf E^s u_f, u_g\rangle = \langle  u_f, \mathbf E^s u_g\rangle,
\end{align*}
where we used the properties of $(\nabla\cdot)^s$, the fact that $\mathbf E^su_f=0$ in $\Omega$ and the self-adjointness of $\mathbf E^s$. Now following the same computations backwards gives $B_{L,M}(u_f,e_g) = B_{L,M}(u_g,e_f)$.
\end{Rem}

With the well-posedness of the direct problem, we can now define the DN map. Consider first the abstract trace space $X:=H^s(\R^n)/\widetilde{H}^s(\Omega)$. It is such that two functions $f_1, f_2\in H^s(\R^n)$ belong to the same equivalence class if and only if they agree in $\Omega_e$, in the sense that $f_1-f_2 \in \widetilde H^s(\Omega)$. If~$\Omega$ happens to be Lipschitz, then it has been proved in ~\cite[p.463]{GSU20} that $X=H^s(\Omega_e)$.

\begin{Lem}[DN map]\label{DN-map}
Let $s\in (0,1)$, and assume that $\Omega\subset\mathbb R^n$ is a bounded open set. There exists a continuous, self-adjoint linear map 
$$ \Lambda_{L,M} : X \rightarrow X^* \qquad\mbox{defined by}\qquad \langle \Lambda_{L,M}[f],[g]\rangle := B_{L,M}(P_{L,M}f,g),$$ 
where $f,g\in H^s(\mathbb R^n)$. 
\end{Lem}

\begin{proof}
The proof is quite standard, and it follows the arguments presented in \cite{GSU20, CMRU20, C21}. The DN map $\Lambda_{L,M}$ is well-defined because of the well-posedness of the direct problem: in fact, since $[f]=f+\widetilde H^s(\Omega)$, we have $P_{L,M} f' = P_{L,M} f$ for all $f'\in[f]$. Moreover, $B_{L,M}(P_{L,M} f,g')=B_{L,M}(P_{L,M} f,g)$ for all $g'\in [g]$ by the definition of the Poisson operator. The boundedness of the bilinear form and the well-posedness estimate give the continuity of $\Lambda_{L,M}$. Finally, the self-adjointness of $\Lambda_{L,M}$ follows from Remark \ref{that-one-property}:
$$ \langle\Lambda_{L,M}[f],[g]\rangle = B_{L,M}(P_{L,M}f,e_g) = B_{L,M}(P_{L,M}g,e_f) = \langle\Lambda_{L,M}[g],[f]\rangle. $$
\end{proof}

\section{The fractional Liouville reduction}\label{sec-liou}
In this Section we show an equivalence between our original problem \eqref{original-problem} and a Schr\"odinger-like problem in which the nonlocal part does not depend on the coefficients. In analogy to the classical transformation from the conductivity equation to Schr\"odinger's, we call this procedure \emph{fractional Liouville reduction} (see also \cite{C20}). The reason of such transformation will be made clear in Remark \ref{why-though}. 
\vspace{3mm}

Recall that every vector valued  $u\in H^s(\mathbb R^n)$ admits a \emph{Helmholtz decomposition}, i.e., it can be written as $u=\nabla\phi + F$, where $\phi$ is the Newtonian potential of $\nabla \cdot u$ and $\nabla\cdot F=0$. See e.g. \cite{DP09} and references therein.

This allows us to define an operator $(\cdot) ' : u \mapsto u'$ for all $u\in H^s(\mathbb R^n) \times H^s(\mathbb R^n)$ such that, if the Helmholtz decomposition of $u$ is $$ u = \begin{pmatrix}u_1 \\ u_2 \end{pmatrix} = \begin{pmatrix}\nabla\phi_1 + F_1 \\ \nabla\phi_2 + F_2 \end{pmatrix}, \quad \mbox{ with } \nabla\cdot F_1 = \nabla\cdot F_2 =0,$$ then $u'\in H^s(\mathbb R^n) \times H^s(\mathbb R^n)$ is $$u':= \begin{pmatrix}
( 2n+4s+n'+2n's )\nabla\phi_1 + ( 2n+4s+n' )F_1    \\
n( 1+2s )\nabla\phi_2 +  n F_2  
\end{pmatrix}.$$ Recall that $n'= 2(n-2)/n$. The operator $(\cdot)'$ is bounded, with
\begin{equation}\label{prime-estimate}\begin{split}
    \|u'\|_{H^s\times H^s} & \lesssim \sum_{j=1}^2(\|\nabla\phi_j\|_{H^s} + \|F_j\|_{H^s}) \leq  \|u\|_{H^s\times H^s}.
\end{split}\end{equation} 
Moreover, if we define the differential operator $D$ acting as $$ D \begin{pmatrix} u_1 \\ u_2
\end{pmatrix} := \begin{pmatrix} D_1u_1 \\ D_2u_2
\end{pmatrix} := -\begin{pmatrix}
d_1\Delta u_1 + d_2 \nabla\nabla\cdot u_1  \\
 d_3\Delta u_2 +d_4\nabla\nabla\cdot u_2 
\end{pmatrix} :=
-\begin{pmatrix}
(2n+4s+n')\Delta u_1 + 2n's \nabla\nabla\cdot u_1  \\
 n\Delta u_2 +2ns\nabla\nabla\cdot u_2
\end{pmatrix}   ,$$ we can compute
\begin{align*}
   - Du & = \begin{pmatrix}
\left((2n+4s+n')\Delta + 2n's \nabla\nabla\cdot\right)(\nabla\phi_1 + F_1)  \\
\left(n\Delta +2ns\nabla\nabla\cdot\right)(\nabla\phi_2 + F_2)
\end{pmatrix}\\ & = \begin{pmatrix}
( 2n+4s+n'+2n's )\Delta\nabla\phi_1 + ( 2n+4s+n' )\Delta F_1    \\
n( 1+2s )\Delta\nabla\phi_2 +  n \Delta F_2  
\end{pmatrix} \\ & = \Delta\begin{pmatrix}
( 2n+4s+n'+2n's )\nabla\phi_1 + ( 2n+4s+n' )F_1    \\
n( 1+2s )\nabla\phi_2 +  n F_2  
\end{pmatrix} = \Delta u',
\end{align*}
which shows that $Du = -\Delta u'$. In particular, $(-\Delta)^{s-1}Du = (-\Delta)^su'$ whenever $u\in H^s(\mathbb R^n)\times H^s(\mathbb R^n)$.

\begin{Pro}[Fractional Liouville reduction]\label{liouville}
Let $L, M \in C^{2s+\eps}(\mathbb R^n)$ satisfy assumptions (A1)-(A3), and assume $f\in C^\infty_c(W)$ with $W\subset \Omega_e$ open and bounded. Define $$\Gamma(x):= \left(
\mu(x), k(x)\right), \qquad Q(x):= \frac{\Gamma(x)}{|\Gamma(x)|^2}\cdot (-\Delta)^{s-1}D(\Gamma(x)\otimes Id). $$  If $u\in H^s(\mathbb R^n)$ solves the original problem
\begin{equation}
    \begin{array}{rll}\label{original2}
       \mathbf E^su & =F & \quad \mbox{ in } \Omega \\
       u & = f & \quad \mbox{ in } \Omega_e
    \end{array}
\end{equation}
in the weak sense for some $F\in L^2(\Omega)$, then $w:=\Gamma\otimes u$ solves the transformed problem
\begin{equation}
    \begin{array}{rll}\label{easier-one}
       (-\Delta)^{s-1}Dw -w\cdot Q & = G & \quad \mbox{ in } \Omega \\
       w & = \Gamma\otimes f & \quad \mbox{ in } \Omega_e
    \end{array}
\end{equation}
in the weak sense in $\Gamma\otimes H^s$, where $G:= \Gamma\otimes \frac{(n/2+s)F}{|\Gamma|^2}$. Conversely, if $w\in \Gamma\otimes H^s(\mathbb R^n)$ solves \eqref{easier-one} in the weak sense in $\Gamma\otimes H^s$ for some $G\in \Gamma\otimes L^{2}(\Omega)$, then $u:=\frac{\Gamma\cdot w}{|\Gamma|^2}$ solves \eqref{original2} in the weak sense, where $F=\frac{\Gamma\cdot G}{n/2+s}$.
\end{Pro}

Observe that saying that $w\in \Gamma\otimes H^s(\mathbb R^n)$ is a solution in the weak sense in $\Gamma\otimes H^s$ means that $$B_Q(w, \Gamma\otimes \phi) = \langle G,  \Gamma\otimes \phi\rangle, \qquad\mbox{ for all } \quad \phi\in\widetilde H^s(\Omega).$$
This can be equivalently written as 
\[
\Gamma \cdot ((-\Delta)^{s-1}Dw -w\cdot Q) = \Gamma \cdot G \text{ in $\Omega$}.
\]
\begin{proof}
Because of the assumptions on $L,M$ we have $\Gamma\in C^{2s+\eps}(\mathbb R^n) \times C^{2s+\eps}(\mathbb R^n)$, with $\Gamma(x)=\left( \mu_0 , k_0 \right) =:\gamma$ for all $x\in\Omega_e$. Therefore $$Q(x) = \frac{\Gamma(x)}{|\Gamma(x)|^2}\cdot (-\Delta)^{s-1}D((\Gamma(x)-\gamma)\otimes Id),$$ where $\Gamma(x)-\gamma=:\tilde\Gamma \in C^{2s+\eps}_c(\Omega)\times C^{2s+\eps}_c(\Omega)$. By formula \eqref{mapping2}, we have $(-\Delta)^{s-1}D((\Gamma(x)-\gamma)\otimes Id)\in C^{\eps}(\mathbb R^n) \subset L^\infty(\mathbb R^n)$. Since $\frac{\Gamma(x)}{|\Gamma(x)|^2}$ is also in $L^\infty(\mathbb R^n)$ by property \eqref{smooth}, we conclude that $Q\in L^\infty(\R^n)$.

Let now $u\in H^s(\mathbb R^n)$. In terms of the new symbols introduced in the statement of the Proposition, after a straightforward computation the equality from Lemma \ref{reduction-lemma} can equivalently be written as 
\begin{equation}\label{the-one-formula}
    \Gamma\cdot\left( (-\Delta)^{s-1}D(\Gamma\otimes u) - (\Gamma\otimes u)\cdot Q \right) = (n/2+s)\mathbf E^su .
\end{equation}
Observe that problem \eqref{easier-one} is well-posed in the Hilbert space $\Gamma\otimes H^s(\mathbb R^n)$. In fact, in $(\Gamma\otimes H^s(\mathbb R^n))\times(\Gamma\otimes H^s(\mathbb R^n))$ the bilinear form $$B_Q(v_1,v_2) := \langle (-\Delta)^{s/2}v_1', (-\Delta)^{s/2}v_2 \rangle - \langle v_1\cdot Q, v_2 \rangle$$ is clearly bounded, and it is also coercive by equation \eqref{the-one-formula}:
\begin{align*}
    B_Q(\Gamma\otimes u, \Gamma\otimes u) & = \langle (-\Delta)^{s-1}D(\Gamma\otimes u)- (\Gamma\otimes u)\cdot Q, \Gamma\otimes u \rangle \\ & = \langle \Gamma\cdot\left( (-\Delta)^{s-1}D(\Gamma\otimes u)- (\Gamma\otimes u)\cdot Q\right), u \rangle \\ & = (n/2+s)\langle\mathbf E^s u,u \rangle \\ & =(2n+4s)U^s(u) \geq 0.
\end{align*}
Thus \eqref{easier-one} has a unique solution of the kind $w=\Gamma\otimes v$ for some $v\in H^s(\mathbb R^n)$, provided that {$G \in L^2(\Omega)$.}
Of course one must have $v=f$ in $\Omega_e$. 

If now $u\in H^s(\R^n)$ solves \eqref{original2} and $G:= \Gamma\otimes \frac{(n/2+s)F}{|\Gamma|^2}$, then by \eqref{the-one-formula} $w := \Gamma \otimes u$ solves \eqref{easier-one}.
Conversely, if $w = \Gamma \otimes v$ solves \eqref{easier-one} in $\Gamma \otimes H^s$ and $F=\frac{\Gamma\cdot G}{n/2+s}$, then by \eqref{the-one-formula} $v$ solves \eqref{original2}. Therefore, $u= \frac{\Gamma\cdot w}{|\Gamma|^2} = \frac{\Gamma\cdot (\Gamma\otimes v)}{|\Gamma|^2} = v$ must also solve \eqref{original2}.
\end{proof}

We also define the adjoint bilinear form
$$B^*_Q(u,v):= \langle (-\Delta)^{s/2}u, (-\Delta)^{s/2}v' \rangle - \langle v\cdot Q, u \rangle,$$ which of course shares the same boundedness inequality as $B_Q$ and can similarly be extended to act on $(\Gamma\otimes \widetilde H^s(\Omega))\times(\Gamma\otimes \widetilde H^s(\Omega))$. It is also clear that we have $B_Q(u,v) = B^*_Q(v,u)$. Given the well-posedness in $\Gamma\otimes H^s$ of problem \eqref{easier-one}, we can define the Poisson operator $P_Q$ associating to the exterior datum $\Gamma\otimes f$ the unique solution $w=\Gamma\otimes v$ to the problem \eqref{easier-one} with $G=0$. We can also define the DN map $\Lambda_Q$ in a similar fashion as in our Lemma \ref{DN-map}: 
\begin{align*}
    \langle\Lambda_Q[\Gamma\otimes f_1]&,[\Gamma\otimes f_2]\rangle := B_Q(P_Q(\Gamma\otimes f_1),\Gamma\otimes f_2) \\ & = \langle (-\Delta)^{s/2}(P_Q(\Gamma\otimes f_1))', (-\Delta)^{s/2}(\Gamma\otimes f_2) \rangle - \langle P_Q(\Gamma\otimes f_1)\cdot Q, \Gamma\otimes f_2 \rangle \\ & = \langle \Gamma\cdot \left( (-\Delta)^{s-1}DP_Q(\Gamma\otimes f_1) - P_Q(\Gamma\otimes f_1)\cdot Q \right), f_2 \rangle ,
\end{align*}
and similarly for $\Lambda_Q^*$. With the usual computation (see e.g.\ the analogous result in \cite{CMRU20}), we get $\langle\Lambda_Q[g_1],[g_2]\rangle = \langle[g_1],\Lambda_Q^*[g_2]\rangle$, which motivates the choice of symbols.

\section{The Alessandrini identity}\label{sec-alex}

The most important instruments needed for proving our main theorem are the so called Alessandrini identity and Runge approximation property, which we study in this and the next section. Let us start from a simple Lemma relating the DN maps of the original and transformed problems:
\begin{Lem}[Relation between the DN maps]\label{relation-btw-DNmaps}
Let $f_j \in C^\infty_c(W_j)$ for $j=1,2$, where $W_1, W_2 \subset \Omega_e$ are open, bounded and disjoint. Then the following equation holds:
\begin{equation}
    (n/2+s)\langle \Lambda_{L,M}[f_1],[f_2] \rangle = \langle \Lambda_Q[\Gamma\otimes f_1] , [\Gamma\otimes f_2] \rangle.
\end{equation}
\end{Lem}
\begin{proof}
Let $u_1$ be the unique solution to problem \eqref{original2} corresponding to the exterior value $f_1$, and let $w_1$ be the unique solution to \eqref{easier-one} corresponding to $u_1$ via the fractional Liouville reduction. In light of formula \eqref{the-one-formula}, we can compute
\begin{align*}
    (n/2+s)\langle \Lambda_{L,M}[f_1],&[f_2] \rangle  = (n/2+s)\langle \mathbf E^s u_1, f_2 \rangle  = \langle \Gamma\cdot\left((-\Delta)^{s-1}Dw_1 - w_1\cdot Q \right), f_2 \rangle \\ & =\langle (-\Delta)^{s-1}Dw_1 - w_1\cdot Q, \Gamma\otimes f_2 \rangle  = \langle \Lambda_Q[\Gamma\otimes f_1] , [\Gamma\otimes f_2] \rangle.
\end{align*}
Therefore, complete knowledge of the DN map $\Lambda_{L,M}$ is equivalent to knowledge of the DN map $\Lambda_Q$ on functions $g_j$ of the kind $\Gamma\otimes f_j$.
\end{proof}

Next, we state and prove the Alessandrini identity:
\begin{Lem}[Alessandrini identity]\label{alex}
Let $\Omega\subset\mathbb R^n$ be a bounded open set and $s\in(0,1)$. Let $L_j, M_j$ for $j=1,2$ be two sets of Lam\'e parameters satisfying assumptions (A1)-(A3), corresponding to $\Gamma_j, Q_j$ for $j=1,2$ through the fractional Liouville reduction. Assume $\Gamma_1(x)=\Gamma_2(x)=:\gamma$ for all $x\in\Omega_e$ and that the relative Poisson ratios of $(L_1,M_1)$ and $(L_2,M_2)$ agree in $\mathbb{R}^n$, i.e.\ $(L_1,M_1) \sim (L_2,M_2)$. Then the following integral identity holds for all $f_1, f_2 \in C^\infty_c(\Omega_e)$
$$ (n/2+s)\langle (\Lambda_{L_1,M_1}-\Lambda_{L_2,M_2})[f_1],[f_2]\rangle = \langle u_1\cdot (Q_1-Q_2), u_2^* \rangle, $$
where $u_1 := P_{Q_1}(\gamma\otimes f_1)$ and $u_2^* := P_{Q_2}^*(\gamma\otimes f_2)$. 
\end{Lem}

\begin{Rem}\label{gauge}
Let $\nu$ be the Poisson ratio corresponding to the Lam\'e parameters $M,K$. By the definition of $\nu$ and Lemma \ref{square-root}, we see that there is a one-to-one correspondence between $\nu$ and the ratio $r:=\mu/k$ of the Lam\'e parameters $\mu,k$ of the square root of the stiffness tensor. Thus $(L_1,M_1)\sim(L_2,M_2)$ if and only if $\mu_1/k_1 = \mu_2 / k_2$, that is, if and only if $\Gamma_1=\rho\Gamma_2$ for some fixed function $\rho$. We say that in this case $\Gamma_1$ and $\Gamma_2$ are themselves in gauge, and we indicate this by $\Gamma_1\sim \Gamma_2$. 
\end{Rem}

\begin{proof}[Proof of Lemma \ref{alex}]
The proof is a computation following from Lemma \ref{relation-btw-DNmaps}:
\begin{align*}
    (n/2+s)\langle (\Lambda_{L_1,M_1}&-\Lambda_{L_2,M_2})[f_1],[f_2]\rangle \\ & = (n/2+s)(\langle \Lambda_{L_1,M_1}[f_1],[f_2]\rangle - \langle \Lambda_{L_2,M_2}[f_1],[f_2]\rangle ) \\ & = \langle \Lambda_{Q_1}[\Gamma_1\otimes f_1] , [\Gamma_1\otimes f_2] \rangle-\langle \Lambda_{Q_2}[\Gamma_2\otimes f_1] , [\Gamma_2\otimes f_2] \rangle \\ & = \langle \Lambda_{Q_1}[\gamma\otimes f_1] , [\gamma\otimes f_2] \rangle -\langle [\gamma\otimes f_1] , \Lambda_{Q_2}^*[\gamma\otimes f_2] \rangle \\ & =B_{Q_1}(P_{Q_1}(\gamma\otimes f_1), {\Gamma_2}\otimes f_2)- B_{Q_2}^*(P_{Q_2}^*(\gamma\otimes f_2), {\Gamma_1} \otimes f_1) \\ & = B_{Q_1}(u_1, u_2^*)-B_{Q_2}^*(u_2^*, u_1) \\ & = \langle u_1 \cdot (Q_1-Q_2), u_2^* \rangle.
\end{align*}
Here we have used the fact that $\Gamma_1 \sim \Gamma_2$ in order to deduce that $B_{Q_1}(u_1, {\Gamma_2} \otimes f_2) = B_{Q_1}(u_1, u_2^*)$, and similarly for the other term. In fact, this will be true as soon as
$$ B_{Q_1}(u_1, \Gamma_2\otimes v_2) =0 $$ for all $v_2 \in \widetilde H^s(\Omega)$, which is granted by the fact that $u_1$ is a weak solution in $\Gamma_1 \otimes H^s$and $\Gamma_2 = \rho \Gamma_1$.
\end{proof}

Given that the right hand side of the Alessandrini identity from Lemma \ref{alex} only contains the difference of the transformed potentials $Q_1$ and $Q_2$, we can at most hope to recover $Q$ from the complete knowledge of the DN map $\Lambda_{L,M}$. This suggests that we may encounter a gauge invariance for our inverse problem: if many different couples of Lam\'e parameters $(L,M)$ give rise to the same transformed potential $Q$, they will remain indistinguishable. Thus we are now left with two tasks: to find appropriate {solutions} to use in the Alessandrini identity which will let us recover information about $Q$, and to study the relative gauge. These problems will be considered in the coming sections. 

\section{Runge approximation property and proof of the main Theorem}\label{sec-proof}

Because of the particular exterior values associated to the solutions appearing in our Alessandrini identity, we do  not need to prove a full Runge approximation property in the sense of \cite{GSU20} or \cite{CMRU20}. We rather need only the following result:

\begin{Lem}[Runge approximation property]\label{better-runge-2}
Let $\Omega, W\subset\mathbb R^n$ be bounded open sets such that $W\subset\Omega_e$, and assume $s\in(0,1)$. Define $$\mathcal R := \left\{ \frac{\Gamma\cdot P^*_{Q}(\Gamma\otimes f)}{|\Gamma|^2}-f :  f\in C^\infty_c(W) \right\} \subset \widetilde{H}^s(\Omega).$$ Then the set $\{ w|_{\Omega} : w \in \mathcal{R} \}$ is dense in $L^2(\Omega)$. The same result holds when we substitute $P_Q$ to $P_Q^*$ in the definition of $\mathcal R$.
\end{Lem}

\begin{proof} By the Hahn-Banach theorem, it is enough to show that any $F\in L^2(\Omega)$ such that $\langle F,v\rangle=0$ for all $v\in\mathcal R$ must vanish identically. Fix any $F\in L^2(\Omega)$ with such property, and consider the problem
\begin{equation*}
    \begin{array}{rll}
       (-\Delta)^{s-1}Dw -w\cdot Q & = - \frac{\Gamma\otimes F}{|\Gamma|^2} & \quad \mbox{ in } \Omega \\
       w & = 0 & \quad \mbox{ in } \Omega_e
    \end{array}.
\end{equation*}
It has a unique weak solution in $\Gamma\otimes \widetilde H^s(\Omega)$ of the form $w=\Gamma\otimes \phi$ by Proposition \ref{liouville}. 
Then for any $f\in C^\infty_c(W)$ we have
\begin{align*}
    0& = \langle F,  \frac{\Gamma\cdot P^*_{Q}(\Gamma\otimes f)}{|\Gamma|^2}-f\rangle \\ & = \langle F,  \Gamma\cdot\frac{ P^*_{Q}(\Gamma\otimes f)-\Gamma\otimes f}{|\Gamma|^2}\rangle = \langle \frac{\Gamma\otimes F}{|\Gamma|^2},   P^*_{Q}(\Gamma\otimes f)-\Gamma\otimes f\rangle \\ & = \langle (-\Delta)^{s-1}Dw -w\cdot Q, \Gamma\otimes f \rangle - \langle (-\Delta)^{s-1}Dw -w\cdot Q, P_Q^*(\Gamma\otimes f) \rangle.
\end{align*}
The second term on the right hand side is $$B_Q(w, P_Q^*(\Gamma\otimes f)) = B_Q^*(P_Q^*(\Gamma\otimes f),w)=0,$$
because of the fact that $w\in\Gamma\otimes \widetilde H^s(\Omega)$ and the definition of weak solution to the adjoint problem. Thus we are left with
\begin{multline*}
    0 = \langle (-\Delta)^{s-1}Dw -w\cdot Q, \Gamma\otimes f \rangle \\ = \langle (-\Delta)^{s-1}Dw -w\cdot Q, \gamma\otimes f \rangle =  \langle (-\Delta)^{s-1}Dw, \gamma\otimes f \rangle
\end{multline*}
because of the assumption that the supports of $f,w$ are disjoint. Eventually 
$$ 0 = \langle (-\Delta)^{s-1}Dw, \gamma\otimes f \rangle = \langle (-\Delta)^{s}w', \gamma\otimes f \rangle = \langle (-\Delta)^{s}(\gamma\cdot w'), f \rangle, $$ which by the arbitrariety of $f$ implies $(-\Delta)^{s}(\gamma\cdot w') =0$ in $W$. The UCP for the fractional Laplacian and the exterior datum of $w$ now imply $\gamma\cdot w' \equiv 0$. Thus in $\Omega$ by Lemma \ref{product-properties}
\begin{align*}
    \frac{(\gamma\cdot\Gamma)F}{|\Gamma|^2}&= \gamma\cdot \frac{\Gamma\otimes F}{|\Gamma|^2} = \gamma \cdot (-(-\Delta)^{s-1}Dw +w\cdot Q) \\ & =  -(-\Delta)^{s}(\gamma\cdot w') +\gamma\cdot(w\cdot Q)  =\gamma\cdot(w\cdot Q) \\ & = \gamma\cdot((\Gamma\otimes \phi)\cdot Q) = (\gamma\cdot\Gamma)(\phi\cdot Q),
\end{align*}
and by the positivity of $\gamma\cdot\Gamma$ we get $\phi\cdot Q = F/|\Gamma|^2$ in $\Omega$. Therefore in $\Omega$ $$ (-\Delta)^sw' = w\cdot Q {-} \frac{\Gamma\otimes F}{|\Gamma|^2} = \Gamma\otimes \left( \phi\cdot Q {-} \frac{F}{|\Gamma|^2} \right)=0, $$
which means that $w'$ solves
\begin{equation*}
    \begin{array}{rll}
       (-\Delta)^{s}w' & = 0 & \quad \mbox{ in } \Omega \\
       w' & = 0 & \quad \mbox{ in } \Omega_e
    \end{array}.
\end{equation*}
By the well-posedness of the direct problem for the fractional Laplacian, we deduce $w'\equiv 0$, which entails $w\equiv 0$ and eventually $F\equiv 0$.
\end{proof}

With this result at hand, we can prove our main Theorem.

\begin{proof}[Proof of Theorem \ref{main-theorem}] \textbf{Step 1.}
Given that the known data can always be restricted, we can without loss of generality assume that the sets $W_1$ and $W_2$ are disjoint. Let $f_j\in C^\infty_c(W_j)$ for $j=1,2$, and define $u_1 := P_{Q_1}(\gamma\otimes f_1)$, $u_2 := P_{Q_2}^*(\gamma\otimes f_2)$.  
By the definition of the Poisson operators we have $u_1 = \Gamma_1\otimes v_1$, $u_2^* = \Gamma_2\otimes v_2$ for some $v_1,v_2\in  H^s(\R^n)$. Thus by the Alessandrini identity from Lemma \ref{alex} and Lemma  \ref{product-properties} it holds that
\begin{equation}\begin{split}\label{main-1}
    0 & = (n/2+s)\langle (\Lambda_{L_1,M_1}-\Lambda_{L_2,M_2})[f_1],[f_2]\rangle \\ & = \langle u_1 \cdot (Q_1-Q_2), u_2^* \rangle \\ & = \langle (\Gamma_1\otimes v_1) \cdot (Q_1-Q_2), \Gamma_2\otimes v_2 \rangle \\ & = \langle v_1 \cdot (Q_1-Q_2), (\Gamma_1\cdot \Gamma_2) v_2 \rangle.
\end{split}\end{equation}
 Let now $g_1, g_2$ be any functions belonging to $C^\infty_c(\Omega)$. Using the Runge approximation property from Lemma \ref{better-runge-2}, we can find two sequences $\{f_{j,i}\}_i\subset C^\infty_c(W_j)$, $j=1,2$, such that

$$  v_{1,i} :=  \frac{\Gamma_1\cdot P_{Q_1}(\gamma\otimes f_{1,i})}{|\Gamma_1|^2}= f_{1,i}+ g_1 + r_{1,i} ,\qquad \mbox{with} \quad \|r_{1,i}\|_{L^2(\Omega)} \leq 1/i,$$
and
$$  v_{2,i} :=  \frac{\Gamma_2\cdot P_{Q_2}^*(\gamma\otimes f_{2,i})}{|\Gamma_2|^2}= f_{2,i}+ g_2 + r_{2,i} ,\qquad \mbox{with} \quad \|r_{2,i}\|_{L^2(\Omega)} \leq 1/i.$$

Substituting $v_{1,i}, v_{2,i}$ into \eqref{main-1} gives
\begin{align*}
    0 & =   \langle v_1 \cdot (Q_1-Q_2), (\Gamma_1\cdot \Gamma_2) v_2 \rangle \\ & =  \langle (f_{1,i}+ g_1 + r_{1,i}) \cdot (Q_1-Q_2), (\Gamma_1\cdot \Gamma_2) (f_{2,i}+ g_2 + r_{2,i}) \rangle \\ & =  \langle ( g_1 + r_{1,i}) \cdot (Q_1-Q_2), (\Gamma_1\cdot \Gamma_2) ( g_2 + r_{2,i}) \rangle
\end{align*}
by the support assumptions. Moreover, the terms containing the errors $r_{j,i}$ vanish as $i\rightarrow \infty$, since for example
$$ |\langle r_{1,i}\cdot (Q_1-Q_2), (\Gamma_1\cdot \Gamma_2)r_{2,i} \rangle| \leq \| r_{1,i}\cdot (Q_1-Q_2)\|_{L^2(\Omega)} \|(\Gamma_1\cdot \Gamma_2)r_{2,i}\|_{L^2(\Omega)} \leq C/i^2,$$
and similarly for the other ones. Therefore, we are left with
\begin{align*}
    0 & = \langle  g_1 \cdot (Q_1-Q_2), (\Gamma_1\cdot \Gamma_2)  g_2  \rangle.
\end{align*}
Since $\Gamma_1\cdot \Gamma_2 = \mu_1\mu_2 + k_1k_2 > 0$, by the arbitrariety of $g_1, g_2 \in C^\infty_c(\Omega)$ we obtain $Q_1=Q_2$ in $\Omega$.
\\

\textbf{Step 2.} We are left with the task of proving that $Q_1=Q_2$ in $\Omega$ implies $\Gamma_1=\Gamma_2$. By Remark \ref{gauge}, the assumption $(L_1,M_1)\sim(L_2,M_2)$ already implies that $\Gamma_2 = r\Gamma_1$ for some fixed but unknown function $r$. Thus $\Gamma_2$ solves
 \begin{align*}
      \Gamma_2\cdot (-\Delta)^{s-1}D(\Gamma_2\otimes Id) = |\Gamma_2|^2 Q_1 & \quad\mbox{ in } \Omega, \\
      \Gamma_2 = \gamma & \quad\mbox{ in } \Omega_e,
 \end{align*}
 which implies that $r$ solves
 \begin{align*}
      \Gamma_1\cdot \left((-\Delta)^{s-1}D(\Gamma_1\otimes (r Id)) - (\Gamma_1 \otimes (r Id))\cdot Q_1\right)=0 & \quad\mbox{ in } \Omega, \\
      r = 1 & \quad\mbox{ in } \Omega_e.
 \end{align*} 
 By Remark \ref{matrix-u}, equation \eqref{big-formula} holds when $u=r Id$. Observe that the right hand side of equation \eqref{big-formula} can be rewritten as the left hand side of the above equation in $\Omega$ (see Proposition \ref{liouville}). This lets us deduce that $\mathbf{E}^s_{L_1,M_1}(r Id)=0$ holds in $\Omega$. Because the direct problem for the fractional elasticity equation can be showed to be well-posed in the matrix case as well by the same strategy used in Proposition \ref{well-posedness}, we conclude that it must necessarily be $r\equiv 1$.
\end{proof}

\begin{Rem}
In order to show that without previous knowledge of the Poisson ratio $\nu$ one may indeed incur in a gauge invariance, we analyze the simple case $n=1$.
\vspace{3mm}

Given that it has only one element, the square root of the stiffness tensor is  just $C_{1111}^{1/2} = \lambda+2\mu = k$. The fractional elasticity operator $\mathbf E^s$ becomes the conductivity operator $\mathbf C^s_{k^2}$ in one dimension, and the fractional Liouville reduction given in this work for the former coincides with the reduction shown in \cite{C20} for the latter. Since the operator has evidently lost all information about $\mu$, it will certainly be impossible to recover such function, which proves that at least this gauge invariance is present in the problem of recovering both Lam\'e parameters $(\mu,k)$. However, \cite{C20} also shows that the conductivity (in our current case $k^2$) can be recovered without gauge from exterior data in the form of the DN map. This proves that in one dimension the gauge is in fact limited to the shear modulus $\mu$.
\end{Rem}

\begin{Rem}\label{why-though}
At this point, we can clarify why did we need a fractional Liouville reduction in the first place. In principle, it would have been possible after the definition of $\Lambda_{L,M}$ to na\"ively prove an Alessandrini identity directly for such DN map. We would have obtained
\begin{equation}\label{wrong-alex}
\langle (\Lambda_{L_1,M_1}-\Lambda_{L_2,M_2})[f_1],[f_2]\rangle = \langle (\mathbf E_1^s-\mathbf E_2^s)u_1, u_2^* \rangle,    
\end{equation}
with clear meaning of the involved symbols. However, the operator on the right hand side of \eqref{wrong-alex} is \emph{nonlocal}, as opposed to the \emph{local} multiplication operator we obtained on the right hand side of the Alessandrini identity from Lemma \ref{alex}. This is no minor distinction, as it affects the success of our plan. In fact, when proving the main theorem we would have obtained
$$ 0 = \langle (\Lambda_{L_1,M_1}-\Lambda_{L_2,M_2})[f_{1,k}],[f_{2,k}]\rangle = \langle (\mathbf E_1^s-\mathbf E_2^s)(f_{1,k}+v_1+r_{1,k}), f_{2,k}+v_2+r_{2,k}  \rangle $$
for some fixed $v_1,v_2\in C^\infty_c(\Omega)$ and $f_{j,k}\in C^\infty_c(W_j)$, $j=1,2$. Given the nonlocality of $\mathbf E_1^s-\mathbf E_2^s$, here we could not dismiss the exterior data $f_{j,k}$ by support assumptions. This is quite dangerous, as it is known that as the Runge approximation improves a corresponding lost of control is expected for the exterior data (see e.g. \cite{RS17}). Thus we are prevented from obtaining the wanted result.

Observe that in most previous works (\cite{GSU20}, \cite{CLR18}, \cite{CMRU20} come to mind) the operator on the right hand side of the Alessandrini identity was local as well. In \cite{C20a} a similar, albeit simpler fractional Liouville reduction was needed for the same reason as in our current work. Finally, the recent paper \cite{C21} considers an Alessandrini identity with a nonlocal operator on the right hand side, but this is assumed to be such that a fine control can still be established over the exterior data.
\end{Rem}

\bibliographystyle{alpha}
\bibliography{main-ARXIV.bbl}

\vspace{5mm}
\begin{itemize}
    \item[] Giovanni Covi - \textsc{Institut fur Angewandte Mathematik, Ruprecht-Karls-Universit\"at Heidelberg, Germany} (\texttt{giovanni.covi@uni-heidelberg.de})
    \item[] Maarten de Hoop - \textsc{Department of Computational and Applied Mathematics, Rice University, Houston, TX, USA} (\texttt{mvd2@rice.edu})
    \item[] Mikko Salo - \textsc{Department of Mathematics and Statistics, University of Jyv\"askyl\"a, Finland} (\texttt{mikko.j.salo@jyu.fi})
\end{itemize}

\end{document}